\documentclass{amsart}

\usepackage{mathrsfs}
\usepackage[OT2, T1]{fontenc}
\usepackage{url}
\usepackage{amsmath}
\usepackage{array}
\usepackage{graphicx}
\usepackage{amsfonts}
\usepackage{amssymb}
\usepackage{amstext}
\usepackage{amsthm}
\usepackage{hyperref}
\usepackage{colonequals}
\usepackage{enumitem}
\usepackage[alphabetic,lite]{amsrefs}
\usepackage{cleveref}
\usepackage[all,cmtip]{xy}
\usepackage{fullpage}

\numberwithin{equation}{subsection}

\newtheorem{theorem}{Theorem}
\newtheorem{lemma}[theorem]{Lemma}
\newtheorem{proposition}[theorem]{Proposition}
\newtheorem{corollary}[theorem]{Corollary}

\theoremstyle{definition}
\newtheorem{defn}[theorem]{Definition}

\theoremstyle{remark}
\newtheorem{rem}[theorem]{Remark}
\newtheorem*{claim}{Claim}
\newtheorem{para}[theorem]{}

\newcommand{\pdiv}{\mathscr{G}}
\newcommand{\bB}{\mathbb{B}}
\newcommand{\bC}{\mathbb{C}}
\newcommand{\bD}{\mathbb{D}}

\newcommand{\bF}{\mathbb{F}}

\newcommand{\bH}{\mathbb{H}}
\newcommand{\bI}{\mathbb{I}}
\newcommand{\bJ}{\mathbb{J}}

\newcommand{\bN}{\mathbb{N}}
\newcommand{\bP}{\mathbb{P}}
\newcommand{\bQ}{\mathbb{Q}}
\newcommand{\bR}{\mathbb{R}}

\newcommand{\bZ}{\mathbb{Z}}

\newcommand{\cA}{\mathcal{A}}
\newcommand{\cB}{\mathcal{B}}

\newcommand{\cD}{\mathcal{D}}
\newcommand{\cE}{\mathcal{E}}
\newcommand{\cF}{\mathcal{F}}

\newcommand{\cH}{\mathcal{H}}

\newcommand{\cM}{\mathcal{M}}

\newcommand{\cO}{\mathcal{O}}

\newcommand{\cS}{\mathcal{S}}
\newcommand{\cT}{\mathcal{T}}

\newcommand{\cZ}{\mathcal{Z}}

\newcommand{\fa}{\mathfrak{a}}
\newcommand{\fd}{\mathfrak{d}}
\newcommand{\fp}{\mathfrak{p}}

\newcommand{\bcH}{\overline{\cH}^{\rm{tor}}}

\newcommand{\homg}{\overline{\omega}}
\newcommand{\sa}{\textrm{univ-sa}}

\newcommand{\cris}{{\mathrm{cris}}}

\newcommand{\cm}{{\mathrm{CM}}}
\newcommand{\pet}{{\mathrm{Pet}}}

\newcommand{\inj}{\hookrightarrow}

\DeclareMathOperator{\GL}{GL}
\DeclareMathOperator{\SL}{SL}
\DeclareMathOperator{\GSp}{GSp}

\DeclareMathOperator{\SO}{SO}

\DeclareMathOperator{\Gal}{Gal}
\DeclareMathOperator{\End}{End}
\DeclareMathOperator{\Hom}{Hom}
\DeclareMathOperator{\Aut}{Aut}

\DeclareMathOperator{\Res}{Res}
\DeclareMathOperator{\Spec}{Spec}
\DeclareMathOperator{\ad}{ad}
\DeclareMathOperator{\Span}{Span}

\DeclareMathOperator{\Fil}{Fil}
\DeclareMathOperator{\disc}{disc}
\DeclareMathOperator{\diag}{diag}
\DeclareMathOperator{\Nm}{Nm}
\DeclareMathOperator{\Deg}{Deg}
\DeclareMathOperator{\Id}{Id}
\DeclareMathOperator{\Div}{Div}
\DeclareMathOperator{\Length}{Length}
\DeclareMathOperator{\Tr}{Tr}

\DeclareMathOperator{\Li}{Li}

\usepackage[usenames,dvipsnames]{color}  %color comments

\newcommand{\revise}[1]{{\color{Blue}  #1}}

\begin{document}

\title[splitting of abelian surfaces]{Exceptional splitting of reductions of abelian surfaces}

\author{Ananth N. Shankar}
\author {Yunqing Tang}

\maketitle

\begin{abstract}
Heuristics based on the Sato--Tate conjecture suggest that an abelian surface defined over a number field has infinitely many places of split reduction. We prove this result for abelian surfaces with real multiplication. As in \cite{Ch} and \cite{Elk2}, this shows that a density-zero set of primes pertaining to the reduction of abelian varieties is infinite. The proof relies on the Arakelov intersection theory on Hilbert modular surfaces.
\end{abstract}

\section{Introduction}\label{intro}
\subsection{Infinitely many nonsimple reductions of a given abelian surface}
Murty and Patankar conjectured in \cite{MP} that an absolutely simple abelian variety over a number field has absolutely simple reduction for a density one set of primes (up to a finite extension) if and only if its endomorphism ring is commutative. Chavdarov (\cite{Cha}) proved their conjecture in the case of abelian varieties of dimension 2 or 6 whose geometric endomorphism ring is $\bZ$. Conditional upon the Mumford--Tate conjecture, Zywina (\cite{Zyw}) established Murty and Patankar's conjecture in full generality. 

It is natural to inquire whether the set of primes (conjecturally a density zero set!) at which a given abelian variety does not have absolutely simple reduction is finite or infinite. 
Based on the Sato--Tate conjecture for abelian surfaces (see \S\ref{sec_ST}), it is expected that the (density zero) set of places of nonsimple reduction of a simple abelian surface is infinite. 
The main result of this paper is the following: 
\begin{theorem}\label{main}
Let $A$ be an abelian surface over a number field $K$. Suppose that $F \subset \End(A) \otimes \bQ$, where $F$ is a real quadratic field. Then $A$ modulo $v$ is not absolutely simple for infinitely many primes $v$ of $K$.
\end{theorem}

\subsection{A heuristic based on the Sato--Tate conjecture}\label{sec_ST}
The classical Sato--Tate conjecture addresses the distribution of Frobenius elements involved in the Galois representation on the \'etale cohomology of a fixed elliptic curve defined over a number field.
The work of Katz--Sarnak \cite{KS}, Serre \cite{Se12}, and Fit\'e--Kedlaya--Rotger--Sutherland (\cite{FKRS})  generalizes this conjecture to higher dimensional abelian varieties. We focus on the case of abelian surfaces with real multiplication and offer a heuristic which indicates that such surfaces have infinitely many places of nonsimple reduction. 

For simplicity, %assume that $A$ is an abelian surface with real multiplication defined over $\bQ$, which does not have potential complex multiplication. 
assume that $A$ is an abelian surface defined over $\bQ$ such that $\End(A)\otimes \bQ=\End(A_{\overline{\bQ}})\otimes \bQ=F$, a real quadratic field.
For each prime $\ell$ of good reduction for $A$, the characteristic polynomial of the Frobenius endomorphism of $A$ modulo $\ell$ is of the form $x^4 + a_1 x^3 + a_2 x^2 + \ell a_1 x + \ell^2$, with $a_1,a_2 \in \bZ$. The roots of this polynomial come in complex conjugate pairs $\lambda_1, \overline{\lambda_1}, \lambda_2, \overline{\lambda_2}$, and each $\lambda_i$ has absolute value $\ell^{1/2}$. Define $s_{i,\ell} = \frac{\lambda_i + \overline{\lambda_i}}{\sqrt{\ell}}$. The Sato--Tate conjecture for abelian surfaces with real multiplication predicts that the distribution of $(s_{1,\ell},s_{2,\ell}) \in [-2,2] \times [-2,2]$, as $\ell$ varies, converges to the measure on $[-2,2]\times [-2,2]$ defined by the function $(\frac{4}{\pi})^2\sqrt{4-s_1^2}\sqrt{4-s_2^2}$ (for instance, see \cite{Ked}). Assuming fast enough rates of convergence to this measure, the probability that $|s_{1,\ell} - s_{2,\ell}| < \frac{1}{\sqrt{\ell}}$ is approximately  the area of the region $V_{\ell} = \{(s_1,s_2): |s_1 - s_2| < \frac{1}{\sqrt{\ell}}\}\subset [-2,2] \times [-2,2]$, which is approximately $\frac{1}{\sqrt{\ell}}$.\footnote{More precisely, there exist absolute constants $C_1,C_2>0$ such that the probability that $|s_{1,\ell} - s_{2,\ell}| < \frac{1}{\sqrt{\ell}}$ is greater than $C_1/\sqrt{\ell}$ and less than $C_2/\sqrt{\ell}$.}

If $|s_{1,\ell} - s_{2,\ell}| < \frac{1}{\sqrt{\ell}}$, it is easy to see that $s_{1,\ell}$ and $s_{2,\ell}$ must be equal. Then by Honda--Tate theory, $A$ modulo $\ell$ is not simple. This gives a heuristic lower bound ($\approx \frac 1{\sqrt{\ell}}$) for the probability that $A$ mod $\ell$ is not simple\footnote{This does not take into account those primes modulo which $A$ is simple, but not absolutely simple.}. On the other hand, $\displaystyle{\sum_{\ell \text{ prime}}\frac 1 {\sqrt{\ell}}}$ diverges, so $A$ should have infinitely many primes of nonsimple reduction.

\subsection{Related results} 
The Sato--Tate conjectures for elliptic curves and pairs of elliptic curves also suggests that the set of primes in either of the following two situations is infinite:
\begin{enumerate}
\item given an elliptic curve $E$ over a number field, consider primes $v$ such that $E$ mod $v$ is supersingular;
\item given a pair of non-isogenous elliptic curves $E_1,E_2$ over a number field, consider primes $v$ such that the reductions of $E_1,E_2$ mod $v$ become geometrically isogenous.
\end{enumerate}

As in our case, both sets of primes have density zero (after taking a finite extension). Indeed, Serre conjectured that up to taking a finite extension of the field of definition, a given abelian variety over a number field has ordinary reduction at a density-one set of primes. Katz proved Serre's conjecture in the case of elliptic curves and abelian surfaces (\cite[pages 370--372]{Ogus}). 
Sawin (in \cite{Sawin}) made explicit the smallest field extension that is required for abelian surfaces.
The second set also has density zero (after taking a finite extension) by Faltings' isogeny theorem (\cite{F85}). 

Elkies proved (1) in \cite{Elk1,Elk2} when the elliptic curve is defined over a number field with at least one real embedding and Charles proved (2) in \cite{Ch}.
\Cref{main} is an analogue of these two results. Indeed, all three results establish that certain thin sets of primes related to the reduction of abelian varieties are infinite.

\subsection{The strategy of the proof}
The proof of \Cref{main} builds on the idea of the proof of the main theorem in \cite{Ch}, where Charles uses Arakelov intersection theory on the modular curve $X_0(1)$ to prove his result. In our case, we use Arakelov intersection theory on the Hilbert modular surface $\cH$ (see \S\ref{sub_H&HZ} for the precise definition). 

Let $[A] \in \cH(K)$ denote the point determined by $A$. Loosely speaking, the modular curve embeds canonically into $\cH$ (we label its image $\Delta$) and parameterizes the locus of split abelian surfaces (along with the product polarization). 
A natural strategy is to consider the Arakelov intersection of $\Delta$ with Hecke orbits of $[A]$.\footnote{Here we refer to the $1$-cycle, given by taking the Zariski closure of the Hecke orbit of $[A]$ in $\cH$ over $\Spec \bZ$.} The local contribution at a finite prime $v$ is positive precisely when the Hecke orbit of $[A]$ intersects $\Delta$ modulo $v$. The reduction of $A$ modulo $v$ would be geometrically nonsimple for such $v$. 

In our proof, we replace $\Delta$ with a compact Hirzebruch--Zagier divisor $\cT$ (see \S\ref{sub_H&HZ} and \S\ref{sub_Borcherds&ht} for definitions). Hirzebruch--Zagier divisors of $\cH$ have the feature that the rank of the N\'eron--Severi group of an abelian surface $B$ increases if $[B]$ lies on these divisors. This has the consequence that over a finite field, an abelian surface is not absolutely simple if it lies on a special divisor. There are two advantages of using a \emph{compact} Hirzebruch--Zagier divisor $\cT$: the first is that we do not have to deal with places of bad reduction for $A$. The second is that we are able to avoid all the cusps of $\cH$ in the archimedean contribution to the global Arakelov intersection. 

In order to prove \Cref{main}, it would suffice to prove that the set of primes which contribute to the intersection is infinite as we vary over infinitely many well-chosen Hecke orbits of $[A]$. There are two steps involved in proving this: 

\begin{itemize}

\item A local step, where we bound the local contribution of the intersection at every place. 

\item A global step, where we compute the growth of the Arakelov intersection of Hecke orbits of $[A]$ with $\cT$. The growth is expressed in terms of the degree of the Hecke operators, and is seen to grow asymptotically faster than the local contributions at each place.

\end{itemize}
Consequently, it follows that more and more primes contribute to this intersection as we vary the Hecke orbit of $A$. 

The abelian surfaces parametrized by Hirzebruch--Zagier divisors have extra special endomorphisms (recalled in \S\ref{sub_H&HZ}).
In the non-archimedean case, our methods are very different from the ones used in \cite{Ch}. For a finite place $v$, we use Grothendieck--Messing theory prove statements about the rate of decay of special endomorphisms of $A[\ell^{\infty}]$ modulo higher and higher powers of $v$. This method avoids the use of CM lifts and can be used in other Arakelov-theoretic situations. We use these results and Geometry-of-numbers arguments to bound the number of special endomorphisms of $A$ modulo powers of $v$. This allows us to prove that the $v$-adic contribution grows asymptotically slower than the global intersection for most of the Hecke orbits that we consider. Indeed, if there were too many Hecke orbits $T_\fp([A])$ having large $v$-adic intersection with $\cT$, then $A$ modulo $v^n$ would have too many special endomorphisms as $n\rightarrow \infty$. 

The arguments used to bound the archimedean contribution are very different from the ones used to bound the finite contributions. A key step in bounding the archimedean contribution is the following statement: for a fixed infinite place, if $[A]$ is close to two Hirzebruch--Zagier divisors, then $[A]$ must be close to their intersection, which is a CM abelian surface. 

In order to prove the global part of our result, it is necessary to relate the global Arakelov intersection of $\cT$ and certain Hecke orbits $T_{\fp}([A])$ (see \S\ref{phecke} for the precise definition) to the intersection of $[A]$ and $\cT$. We accomplish this in two steps: 

\begin{itemize}
\item 
We use Borcherds' theory (briefly recalled after \Cref{cpt_div}) to construct a compact special divisor, whose class in the Picard group of a toroidal compactification of $\cH$ equals, up to multiplying by a constant in $\bZ_{>0}$, the class of the Hodge bundle. Consequently, the global Arakelov intersection of $[A]$ with $\cT$ (endowed with a suitable Hermitian metric),\footnote{Technically speaking, since we will not make $[A]$ into an arithmetic cycle, here by Arakelov intersection, we mean the height of the $1$-cycle $[A]$ with respect to the Arakelov divisor $\cT$, which is endowed with a Hermitian metric by Borcherds' theory.} up to multiplying by a suitable constant, equals the Faltings height of $A$. 
\item 
We relate the Faltings height of $T_\fp([A])$ to the Faltings height of $A$ when $A$ has potentially good reduction at $p$ in \Cref{thm_glb}. This extends a result of Autissier (\cite[Theorem 5.1]{Au}), which only applies to $A$ with potentially ordinary reduction at $p$.
\end{itemize}
It follows that the global intersection number grows faster than any local contribution. Hence, infinitely many primes occur in the intersection of $T_\fp([A])$ and $\cT$ as $p\rightarrow \infty$.

\subsection{Organization of the paper}
In \S\ref{outline}, we recall the definitions of the Hilbert modular surface and the Hirzebruch--Zagier divisors. In \S \ref{arch}, we bound the archimedean contribution. We spend \S \ref{finite} counting special endomorphisms and bounding the non-archimedean contribution. We use Borcherds' theory in \S\ref{sub_Borcherds&ht} to choose a compact Hirzebruch--Zagier divisor and relate the Arakelov intersection number to Faltings height. We also extend Autissier's result to the setting of Hilbert modular surfaces. Finally, we assemble all these results together in \S \ref{proof} to prove \Cref{main}.

\subsection{Notation and conventions}
We use $K$ to denote a number field and let $\cO_K$ be its ring of integers. As in \Cref{main}, we use $F$ to denote a fixed real quadratic field with discriminant $D$; its ring of integers is denoted by $\cO_F$ and $\fd_F$ is its different ideal. 
For $a\in F$, we use $\Nm a$ to denote its $F/\bQ$-norm.

The statement of \Cref{main} is invariant under isogeny. Therefore, we will always assume that $A$ has real multiplication by the maximal order $\cO_F$, and is equipped with an $\fa$-polarization for some (integral) ideal $\fa$ of $\cO_F$; see \cite{Pa}*{\S~2.1 item 2 before Def.~2.1.1} for the definition of an $\fa$-polarization. We may also assume that $A$ has semistable reduction over $K$. For any abelian varieties $B,B'$, we use $\End(B)$ and $\Hom(B,B')$ to denote the endomorphism ring of $B$ and the $\bZ$-module given by homomorphisms from $B$ to $B'$.

Let $\cH$ be the Hilbert modular surface over $\bZ$ associated to $F$ which is the moduli stack of abelian surfaces $B$ with real multiplication by $\cO_F$ and an $\fa$-polarization. This is a Deligne--Mumford stack and we use $[B]$ to denote the point of $\cH$ determined by $B$. Sometimes, we may denote a point of $\cH$ by $[B']$; this means that $B'$ is the abelian surface determines this given point.

We always use mathcal letters to mean the natural extension over certain ring of integers. For example, we use $\cA$ to denote the everywhere semistable semi-abelian scheme over $\cO_K$ such that $\cA_K=A_K$. 

Throughout the text, $v$ means a place of $K$, either archimedean or finite. If $v$ is finite, we use $\bF_v$ to denote its residue field and $e_v$ to denote the degree of ramification of $K$ at $v$. If $\cA$ has good reduction at $v$, we use $\cA_{v,n}$ to denote $\cA$ modulo $v^n$. We always use $p$ to denote a prime number which is totally split in the narrow Hilbert class field of $F$ and denote by $\fp,\fp'$ the two primes ideals of $F$ above $p$. We use $A[p]$ and $A[\fp]$ to denote the $p$-torsion and $\fp$-torsion subgroups of $A$.

%\ananth{do we need this definition?} We use $M_2(R)$ to denote the ring of $2\times 2$ matrices over some ring $R$. 

\subsection*{Acknowledgements}
We thank George Boxer, Francesc Castella, K\k{e}stutis~\v{C}esnavi\v{c}ius, Fran\c{c}ois~Charles, William Chen, Noam~Elkies, Ziyang Gao, Chi-Yun Hsu, Nicholas Katz, Ilya Khayutin, Djordjo Milovic, Lucia Mocz, Peter~Sarnak, William Sawin, Arul Shankar, Jacob Tsimerman, Tonghai Yang, and Shou-Wu Zhang for useful comments and/or discussions. We are also grateful to  K\k{e}stutis~\v{C}esnavi\v{c}ius, Chao Li, Mark Kisin, and Lucia Mocz for very useful comments on previous versions of this paper. We would like to thank Davesh Maulik for pointing a gap in \Cref{numspecial} in an earlier version of this paper. The second author, during her stay at the Institute for Advanced Study, was supported by the NSF grant DMS-1128115 to IAS.

\numberwithin{theorem}{subsection}

\section{Hirzebruch--Zagier divisors and Hecke orbits}\label{outline}
In this section, we first recall the definition of Hirzebruch--Zagier divisors and their properties and then we specify the Hecke orbits that will be used in the rest of the paper.

\subsection{The Hilbert modular surface and the Hirzebruch--Zagier divisors}\label{sub_H&HZ}
%Let $\cS$ be the moduli stack over $\Spec \bZ$ that parametrizes principally polarized abelian surfaces  Let $(\GSp_4,S^{\pm})$ be the Shimura datum associated to $\cS$. 
%Let $G_1$  denote $\Res^F_\bQ \GL_2$. Define $G_2\subset G_1$ to be the subgroup such that for any $\bQ$-algebra $R$, the set $G_2(R)$ consists of matrices with determinant in $R$ (instead of $R\otimes_\bQ F$). The groups $G_1$ and $G_2$ naturally give rise to Shimura data $(G_1,X_1)$ and $(G_2,X_2)$. There is a natural embedding of $\bQ$-groups $G_2\subset \GSp_4$, which induces an embedding\footnote{One needs to choose suitable level structure to ensure that the natural finite morphism is an embedding.} of the Hilbert modular surface $\cH_{\bQ} = Sh(G_2,X_2)$ into $\cS_{\bQ}$ (see \cite{vdG}*{Chp. ~ IX.1}). Let $\cH_1$ denote the canonical integral model of $Sh(G_1, X_1)$. As $G_1$ and $G_2$ have the same adjoint group, the associated connected Shimura varieties are the same and we refer the reader to \cite{vdG}*{Chp.~I} for a detailed discussion on the relation between the group of components and narrow class group $Cl^{+}(F)$. 
%In particular, if $Cl^+(F)$ is trivial, then $\cH=\cH_1$ and the main ideas of the proof manifest themselves already in this case. 

Recall that we use $\cH$ to denote the moduli stack over $\Spec \bZ$ that parametrizes abelian surfaces with real multiplication by $\cO_F$ and an $\fa$-polarization (see \cite{Pa}*{Def.~2.1.1}). It is a Deligne--Mumford stack. A totally positive element $a\in \fa$ gives rise to a polarization on $A$ and a symplectic form $\psi$ on the Betti cohomology group $W=H^1_B(A(\bC),\bQ)$ (here we choose an embedding $K\rightarrow \bC$). The set of $\GSp(W,\psi)(\bR)$-conjugate cocharacters of the Hodge cocharacter (from the Hodge decomposition of $W\otimes \bC$) of $A_{\bC}$ coincides with $\bH_2^{\pm}$, the upper and lower Siegel half plane of genus $2$. Let $G\subset \GSp(W,\psi)$ be the subgroup (over $\bQ$) that commutes with $\cO_F\subset \End(W)$ and let $G_1=\Res^F_\bQ \GL_2$. Then $G$ is naturally isomorphic to the subgroup of $G_1$ such that for any $\bQ$-algebra $R$, the set $G(R)$ consists of matrices with determinant in $R$ (instead of $R\otimes_\bQ F$). The embedding $G\subset \GSp_4$ induces an embedding\footnote{One needs to choose suitable level structure to ensure that the natural finite morphism is an embedding.} of the Hilbert modular surface $\cH_{\bQ} = Sh(G,X)$ into $Sh(\GSp(W,\psi),\bH_2^{\pm})$ (see \cite{vdG}*{Chp. ~ IX.1} and here $X$ is the subset of $\bH_2^{\pm}$ which consists of cocharacters conjugate to the Hodge cocharacter of $A_{\bC}$ under $G(\bR)$).

Let $\bcH$ be a toroidal compactification of $\cH$ as in \cite{R} (see also \cite{Chai}*{\S3}).  
The stack $\bcH$ is regular and proper and $\bcH\backslash \cH$ is a normal crossing divisor (see, for example, \cite{Pa}*{2.1.2, 2.1.3} for the regularity of $\cH$ and \cite{Chai}*{Thm.~3.6, 4.3} for the smoothness of the formal neighborhood of the boundary and the property that the boundary is a normal crossing divisor), and hence the arithmetic intersection theory developed by Burgos~Gil, Kramer, and K\"uhn in \cite{BKK} applies to $\bcH$ (see, for example, \cite{BBK}*{\S1, \S6} for a summary of their theory). 
We will use $[\cA]$ (resp. $[A]$) to denote the unique $\cO_K$-point (resp. $K$-point) of $\cH$ corresponding to $A$ (the stack $\bcH$ being proper allows us to do this).\footnote{In general, one needs to pass to a finite field extension to extend a $K$-point on a proper Deligne--Mumford stack to an $\cO_K$ point; however, since we have assumed that $A$ has a semistable integral model $\cA$ over $\cO_K$, we do not need to pass to a further field extension.} 
\begin{para}\label{def_HZ}
We now summarize some basic facts about $\cH_{\bC}$ and the Hirzebruch--Zagier divisors. The facts discussed here can be found in \cite{vdG}*{Chp.~I, V, IX}, \cite{BBK}*{\S2.3, \S5.1} and \cite{Go}*{Chp.~2}; however, since conventions differ, we will use this subsection to fix our notation. After giving the definition of Hirzebruch--Zagier divisors, we first show that among them, there exists a nonempty compact one and then we give a moduli interpretation of these divisors.

Let $\bH$ denote the upper half plane. The two real embeddings of $F$ induce two embeddings $\sigma_1,\sigma_2:\SL_2(F)\rightarrow \SL_2(\bR)$. The action of $g\in \SL_2(F)$ on $\bH^2$ is given by $\sigma_1(g)$ on the first copy of $\bH$ and by $\sigma_2(g)$ on the second copy. Since our Hilbert modular surface $\cH$ admits a map to $\cS$, we have $\displaystyle \cH(\bC)=\Gamma\backslash \bH^2$, where $\displaystyle \Gamma= \SL_2(F)\cap
\begin{pmatrix}
\cO_F & (\fa\fd_F)^{-1} \\
\fa\fd_F & \cO_F
\end{pmatrix}$ (see, for example, \cite{Go}*{pp.~71}).\footnote{The Hilbert modular surface $\cH$ is connected and hence we may use $\Res^F_\bQ \SL_2$ instead of $G$ to study the complex points. Notice that our lattice is different from the default choice in \cite{BBK}. Using their notation, we work with $\Gamma(\cO_F\oplus \fa\frak{d}_F)$.}

For any $r\in \bZ_{>0}$, we recall the definition of the Hirzebruch--Zagier divisors $T(r)$ in $\cH_{\bC}$ (see for example \cite{vdG}*{V.1.3} and \cite{BBK}*{sec.~2.3}). Let $\gamma'$ denote the $\Gal(F/\bQ)$-conjugate of a given $\gamma\in F$. Consider the lattice 
$$L=\left\{
\begin{pmatrix}
a &  \gamma\\
\gamma' & b
\end{pmatrix}:a\in (D\Nm \fa)\bZ, b\in \bZ, \gamma \in \fa\right\}$$
in the rational quadratic space $$V=\left\{
\begin{pmatrix}
a &  \gamma\\
\gamma' & b
\end{pmatrix}:a,b\in \bQ, \gamma \in F\right\}$$
with the quadratic form given by the determinant. The group $\Gamma$ acts on $V$ via $v.g=(g')^t\cdot v\cdot g$ for $g\in \Gamma, v\in V$ and this action preserves $L$. The quadratic space $V$ is of signature $(2,2)$. One may also view $\cH$ as an orthogonal type Shimura variety defined by $\SO(V)$. The divisor $T(r)$ is defined to be the reduced divisor in $\cH_{\bC}$ whose set of $\bC$-points is the image of\footnote{This is the definition in \cite{BBK}. The lattice in \cite{vdG} differs by a multiple of the scalar matrix $\sqrt{D}\cdot I$ so these two definitions of $T(r)$ coincide.}
$$\bigcup_{M\in L,\, \det(M)=r\Nm \fa}\{(z_1,z_2)\in \bH^2: az_1z_2+\gamma z_1+\gamma'z_2+b=0\}.$$
\end{para}

\begin{proposition}\label{HZdiv}
The divisor $T(r)$ is nonempty if and only if $r\Nm \fa$ modulo $D$ is $-\Nm \gamma$ for some $\gamma\in \fa$. In this case, $T(r)$ is defined over $\overline{\bQ}$ and is either a modular curve or a Shimura curve defined by the indefinite quaternion algebra $\displaystyle \left(\frac{D, -r\Nm \fa}{\bQ}\right)$. If $r$ is not the norm of an ideal of $\cO_F$, then $T(r)$ is a Shimura curve, and hence compact.
\end{proposition}
\begin{proof}
The first assertion follows from the definition of a Hirzebruch--Zagier divisor. By the discussion on \cite{vdG}*{pp.~89--90}, the divisor $T(r)$ is the union of Shimura curves defined by the quaternion algebra mentioned above and hence is defined over $\overline{\bQ}$. %By the theory of canonical models (see, for example, \cite{Mil}*{II.5}), $T(r)$ and the embedding of $T(r)$ into $\cH$ are defined over $\bQ$. 
The last assertion follows from \cite{vdG}*{Chp.~V, 1.7}.
\end{proof}

\begin{corollary}\label{good_r}
Let $q$ denote a rational prime inert in $F$. Then the divisor $T(qD)$ is non-empty and compact. 
\end{corollary}
\begin{proof}
%\end{corollary}
%\begin{proof}
As $q$ is inert, $qD$ is not the norm of an ideal of $\cO_F$. Further, $qD$ is $0$ modulo $D$. It follows from \Cref{HZdiv} that $T(qD)$ is compact and nonempty.
%We can take $r_0=q\cdot D$. 
\end{proof}

Hirzebruch--Zagier divisors parametrize abelian surfaces with extra special endomorphisms. After recalling the definition of special endomorphisms, we give a sketch of the proof of this fact (see \Cref{HZ=spEnd}),\footnote{We only deals with $T(Dr)$ since these are the divisors that we will use in the proof of \Cref{main}. However, after minor modification, the proof shows that any $T(r)$ parametrizes abelian surfaces with an extra special endomorphism.} which may be well known to experts. From now on, $B$ denotes an abelian surface over some $\bZ$-algebra with an $\fa$-polarization and $\cO_F\subset \End(B)$. We fix a totally positive element in $\fa\cap \bQ$ and this provides a fixed polarization on $B$ and we use this polarization to define the Rosati involution $(-)^*$ on $\End(B)\otimes \bQ$. 
\begin{defn}[see also \cite{KR}*{Def.~1.2}]
An $s\in \End(B)$ is a \emph{special endomorphism} if $a\circ s=s\circ a'$ for all $a\in \cO_F\subset \End(B)$ and $s^*=s$.
\end{defn}
All the special endomorphisms of $B$ form a sub-$\bZ$-module of $\End(B)$. It is well known that the rank of this submodule is at most 4. The following lemma recalls the discussion after \cite{KR}*{Def.~1.2}. 
\begin{lemma}
For a special endomorphism $s$, there is a $Q(s)\in \bZ$ such that $s\circ s=Q(s)\cdot \Id_B$ and hence also $\Deg s=Q(s)^2$. The function $Q$ is a positive definite quadratic form on the $\bZ$-module of special endomorphisms of $B$. 
\end{lemma}

%\ananth{Add lemma listing number of special endomorphisms of ordinary and supersingular point}

The $\bQ$-vector space in $\End(B)\otimes \bQ$ generated by special endomorphisms depends on the choice of the polarization on $B$. However, there are natural isomorphisms between the $\bQ$-vector spaces of special endomorphisms defined by different polarizations and the quadratic forms coincide up to multiplying by a fixed scalar determined by the polarizations.

\begin{lemma}\label{HZ=spEnd}
The Hirzebruch--Zagier divisor $T(Dr)$ defined in \cref{def_HZ} is the locus of $\cH_{\overline{\bQ}}$ where the abelian surface has a special endomorphism $s$ with $Q(s)=r\Nm \fa$. In particular, the degree of the endomorphism $s$ is $(r\Nm \fa)^2$.
\end{lemma}
\begin{proof}
We only need to check the statement over $\bC$. 
Given a point in $\cH(\bC)$ corresponding to $(z_1,z_2)\in \bH^2$, it corresponds to an abelian surface $B$ with $B(\bC)=\bC^2/\cO_F(z_1,z_2)+(\fa\fd_F)^{-1}$. The Riemann form $E$ on $H_1(B,\bZ)$ is, up to multiplying by a constant $\in \bQ_{>0}$, the pull back of the standard alternating form ($\Tr_{F/\bQ}$ of $\begin{pmatrix}
0 &  1\\
-1 & 0
\end{pmatrix}$) on $\cO_F\oplus (\fa\fd_F)^{-1}$ via the isomorphism $\cO_F(z_1,z_2)+(\fa\fd_F)^{-1}\cong \cO_F\oplus (\fa\fd_F)^{-1}$; see, for example, \cite{vdG}*{p.~208} and \cite{BBK}*{the discussion after Thm.~5.1}. Any endomorphism of $B$ is given by the induced map on $B(\bC)$ of some $\bC$-linear map on $\bC^2$. For any endomorphism $s$, the condition $f\circ s=s\circ f'$ for all $f\in \cO_F$ is equivalent to the condition that the $\bC$-linear map corresponding to $s$ is of the form $(1,0)\mapsto (0,\alpha'z_2+\beta'),(0,1)\mapsto (\alpha z_1+\beta)$ where $\alpha \in \fa\fd_F, \beta \in  \cO_F$. This linear map gives rise to an endomorphism of $B$ if and only if the image of $(z_1,z_2)$ is in the period lattice. In other words, there exists $\nu\in \cO_F, \delta\in (\fa\fd_F)^{-1}$ such that $(z_2(\alpha z_1+\beta),z_1(\alpha'z_2+\beta'))=(\nu z_1+\delta,\nu' z_2+\delta')$.

For every component of $T(D r)$, there exists $M\in L$ in \cref{def_HZ} satisfying $\det(M)=Dr\Nm \fa$. Write $M=
\begin{pmatrix}
a &  \gamma\\
\gamma' & b
\end{pmatrix}$ where $a\in (D\Nm \fa)\bZ, b\in \bZ, \gamma \in \fa$. Since $D\Nm \fa \mid \gamma\gamma'$, we have $\frac \gamma {\sqrt{D}}\in \fa\subset \cO_F$. Moreover $\frac a {\sqrt{D}}\in (\Nm \fa)\fd_F\subset \fa \fd_F$. We take $\alpha=\frac a {\sqrt{D}}$ and $\beta=\frac {\gamma'} {\sqrt{D}}$. Given $(z_1,z_2)\in \bH^2$ such that $az_1z_2+\gamma z_1+\gamma' z_2+b=0$, we have 
$$\alpha z_1z_2+\beta z_2=\beta'z_1-\frac b{\sqrt{D}},\quad \alpha' z_1z_2 +\beta' z_1=\beta z_2 + \frac b{\sqrt{D}}.$$
Hence $(1,0)\mapsto (0,\alpha' z_2+\beta'),(0,1)\mapsto (\alpha z_1+\beta)$ is an endomorphism $s$ with $f\circ s=s\circ f'$ for all $f\in \cO_F$.

To check that $s=s^*$, it is equivalent to check that for any $u,v\in H_1(B,\bZ)$, one has $E(su,v)=E(u,sv)$. Since we have already checked that $f\circ s=s\circ f'$ for all $f\in \cO_F$, one only needs to check the above equality for $u,v\in \{e_1=(z_1,z_2), e_2=(1,1)\}\subset \bC^2$. By construction, $se_1=\beta'e_1-\frac b{\sqrt{D}}e_2$ and $se_2=\alpha e_1+\beta e_2$ and then we conclude by the fact that $\Tr_{F/\bQ}\beta=\Tr_{F/\bQ}\beta',\,\Tr_{F/\bQ}-\frac b{\sqrt{D}}=0$, and $\Tr_{F/\bQ}\alpha=0$.
Moreover, on $\bC^2$, the composite $s\circ s=\frac {\det(M)}{D}\cdot \Id_{\bC^2}$. Hence $s$ is a special endomorphism with $Q(s)=r\Nm \fa$.

On the other hand, the moduli space of $B$ with a special endomorphism is $1$-dimensional. Hence the two conditions $z_2(\alpha z_1+\beta)=\nu z_1+\delta$ and $z_1(\alpha'z_2+\beta')=\nu' z_2+\delta'$ are linearly dependent. Hence either $\alpha,\delta\in \bQ,\,\beta=-\nu'$ or $\alpha\cdot\sqrt{D}, \,\delta\cdot\sqrt{D}\in \bQ,\, \beta=\nu'$. In the first case, we have 
$$\alpha z_1z_2+\beta'z_1+\beta z_2-\delta=0, \beta\beta'+\alpha\delta=r>0.$$
In this case, there is no $(z_1,z_2)$ satisfying the above condition (see for example \cite{vdG}*{V.4}).
In the second case, take $a=\alpha\cdot\sqrt{D}, b=-\delta\cdot\sqrt{D}, \gamma=\beta'\cdot\sqrt{D}$. Then $M=
\begin{pmatrix}
a & \gamma\\
\gamma' & b
\end{pmatrix}\in L$ and hence $[B]\in T(Dr)$. 
\end{proof}

Let $\cT(r)$ be the Zariski closure of $T(r)$ in $\bcH$ over $\Spec \bZ$.
\begin{corollary}\label{notsimple}
Assume $T(Dr)$ is compact. Then for any finite place $v$, the points on $\cT(Dr)_{\overline{\bF}_v}$ correspond to abelian surfaces which are not absolutely simple. The abelian surfaces parametrized by $\cT(Dr)$ admit a special endomorphism $s$ such that $Q(s)=r\Nm \fa$.
\end{corollary}
\begin{proof}
Since $T(Dr)$ is compact, every point parametrized by it has potentially good reduction.
For any given point parametrized by $\cT(Dr)_{\bF_v}$, let $[\cB]$ be a lift of the point on $\cT(Dr)$ over $\cO_K$ for some number field $K$.  By \Cref{HZ=spEnd}, the N\'eron--Severi rank of $\cB_K$ is $3$ and hence the N\'eron--Severi rank of $\cB_{\bF_v}$ is at least $4$. By the classification of the endomorphism algebra of abelian varieties, we see that the N\'eron--Severi rank of a geometrically simple abelian surface is at most 2 and then conclude that $\cB_{\bF_v}$ is not geometrically simple. The last assertion follows from \Cref{HZ=spEnd} and the fact that the canonical reduction map $\End(\cB_K)\rightarrow \End(\cB_{\bF_v})$ is injective. 
\end{proof}

\subsection{Hecke orbits}\label{Hecke}
The idea of the proof of \Cref{main} is to show that the corresponding Hecke orbits of $[\cA]$ intersect certain Hirzebruch--Zagier divisors at more and more places of $K$ as one varies over certain well-chosen Hecke operators.\footnote{We view every Hecke orbit as a horizontal divisor over $K$, so the arithmetic intersection number is a sum over the finite places of $K$.} In this subsection, we specify the Hecke orbits which we will use later. %and use Autissier's idea in \cite{Au} to study the sum of the Faltings heights on these Hecke orbits.

\begin{para}\label{phecke}
Recall that $p$ is a prime which splits completely in the narrow Hilbert class field of $F$ and $(p)=\fp\fp'\subset \cO_F$. Hence $\fp=(\lambda),\,\fp'=(\lambda')$ with $\lambda,\lambda'\in F$ totally positive and $\lambda\lambda'=p$. 
Let $G^{\ad}_1$ be the adjoint group of $G_1=\Res^F_{\bQ}\GL_2$. 
We denote by $G^{\ad}_1(\bR)_1$ the image of $G_1(\bR)$ in $G^{\ad}_1(\bR)$ and let $G^{\ad}_1(\bQ)_1$ be $G^{\ad}_1(\bQ)\cap G^{\ad}_1(\bR)_1$.
Since $\lambda$ is totally positive, the image of the diagonal matrix $g_\fp:=\diag(1,\lambda)$ under $G_1\rightarrow G_1^{\ad}$ lies in $G_1^{\ad}(\bQ)_1$, so it induces a correspondence $T_{\fp}$ on $\cH_{\bZ[1/p]}$ (defined in \cite{D77}; see also \cite{Kisin}*{sec.~3.2}).\footnote{The definition of $T_\fp$ depends on the choice of $\lambda$ if we do not pass to a certain finite quotient of $\cH$. However, there are only finitely many choices: let $U$ be the unit group of $\cO_F$ and $U^+$ the subgroup of totally positive units; then the number of choices is $\#U^+/U^2$. Hence we will not specify our choice of $\lambda$ as it does not affect the arguments in this paper.} The following lemma provides a moduli interpretation of $T_{\fp}$.
\end{para}

\begin{lemma}\label{Tpmoduli}
We have $\#T_{\fp}[A]=p+1$. Over $\bZ[1/p]$, the set $T_{\fp}[\cA]$ consists of those points on $\cH$ that correspond to a quotient of $A$ by an order $p$ subgroup in $A[\fp]$ endowed the with induced $\cO_F$-structure and a suitable $\fa$-polarization.\footnote{The choice of $\lambda$ determines the polarization.}
\end{lemma}
\begin{proof}
The first assertion follows from the definitions: $\# T_{\fp}[A]=\# \Gamma/(g_\fp^{-1}\Gamma g_\fp\cap \Gamma)= \#\bP^1(\bF_p)=p+1.$

For the second assertion, since the correspondence $T_{\fp}$ is \'etale, we only need to show the same statement for $T_{\fp}[A]$ over $\bC$ for a fixed embedding of $K$ into $\bC$.
On the one hand, as $\cO_F$ acts on $A[\fp]$ via $\cO_F/\fp\cong \bF_p$ and $\bZ\subset \cO_F$ surjects onto $\bF_p$, any subgroup of $A[\fp]$ is $\cO_F$-invariant. Therefore, any quotient of $A$ by an order $p$ subgroup of $A[\fp]$ has the induced $\cO_F$-structure. Moreover, by \cite{BBK}*{Lem.~5.9}, any such quotient of $A$ is $\cO_F$-polarizable.  

On the other hand, by \cite{vdG}*{p.~208}, if a point $(z_1,z_2)\in \bH^2$ corresponds to $A_{\bC}$, then $A(\bC)$ is isomorphic to $\bC^2/(\cO_F(z_1,z_2)+(\fa\fd_F)^{-1})$. Then $\diag(1,\lambda)z$ corresponds to $\bC^2/(\cO_F(z_1/\lambda, z_2/\lambda')+(\fa\fd_F)^{-1})$, so the kernel of the isogeny defined by $\diag(1,\lambda)$ is contained in $\ker(\lambda)=A[\fp]$. On the quotient $\bC^2/(\cO_F(z_1/\lambda, z_2/\lambda')+(\fa\fd_F)^{-1})$, the $\cO_F$-structure is the induced one and the choice of $\lambda$ determines the $\cO_F$-polarization.
Since the other elements in $T_{\fp}$ differ from $\diag(1,\lambda)$ by the action of some element in $\Gamma$ (on $\bH^2$ and on $A[\fp]$), the set $T_{\fp}$ injects into the set of order $p$ subgroups of $A[\fp]$. Since both sets have cardinality $p+1$, this is in fact a bijection.
\end{proof}

\section{Archimedean places and equidistribution of Hecke orbits}\label{arch}

Let $\Psi$ be a meromorphic Hilbert modular form of parallel weight $k$ over $\overline{\bQ}$ such that $\Div(\Psi)$ in $\cH_{\bQ}$ is given by $\sum_{c\in \bI} c_rT(r)$, where $k\in \bN_{>0}$, $\bI$ is a finite set, $c_r\in \bZ$, $T(r)$ is compact and $D|r$ for all $r\in \bI$. In the proof of \Cref{main}, we will use \Cref{cpt_div} to construct such meromorphic Hilbert modular form. %In particular, $\bI$ is a finite set and $T(r)$ is compact for any $r\in \bI$. 
We assume that $\End(A_{\overline{K}})=\cO_F$ and hence $T_{\fp}([A])$ does not intersect $T(r)$ in characteristic zero. This is the key case in the proof of \Cref{main}. Fix an embedding $\sigma:\overline{K}\rightarrow \bC$. Given an abelian surface $B$ corresponding to a point $[B]$ on $\cH_{\overline{\bQ}}$, we use $\sigma([B])$ to denote the corresponding $\bC$-point on $\cH$ via base change by $\sigma$. 

We set $||\Psi(z)||_{\rm{Pet}}=|\Psi(z_1,z_2)(\Im z_1)^{k/2}(\Im z_2)^{k/2}|$, where $z=(z_1,z_2)\in \bH^2$. This norm is well-defined outside (the preimage of) $\bigcup_{r\in \bI}T(r)$ and invariant under $\Gamma$ (defined in \cref{def_HZ}) and hence we will also view $||\Psi||_{\rm{Pet}}$ as a function on $\cH_\bC\backslash \bigcup_{r\in \bI}T(r)$. The real analytic function $-\log ||\Psi||_{\rm{Pet}}$ is a Green function for $\sum_{c\in \bI} c_rT(r)$ and endows it with the structure of an arithmetic divisor $\widehat{\sum_{c\in \bI} c_rT(r)}$.

The goal of this section is to show that for most $p$ as in \cref{phecke}, the archimedean contribution (in the height of $T_\fp([\cA])$ with respect to the arithmetic divisor $\widehat{\sum_{c\in \bI} c_rT(r)}$; we will discuss this height in detail in \S\ref{sub_Borcherds&ht})
$$-\sum_{[B]\in T_{\fp}[A]} \log ||\Psi(\sigma([B]))||_\pet=o(p\log p) \text{ as }p\rightarrow \infty.$$ The equidistribution theorems for Hecke orbits on Shimura varieties reduces this goal to a suitable upper bound for $-\log ||\Psi(\sigma([B]))||_\pet$ for all $[B]\in T_\fp[A]$ which is valid for most $p$. The proofs are inspired by Charles' treatment in the case of the modular curve. Recall that all $T(r)$ ($r\in \bI$) are compact, so we avoid dealing with estimates around the cusps. 

Throughout this section, $p$ is a prime as in \cref{phecke} and $N_i(*),C_i(*)$ denote constants only depending on $*$. In particular, if there is no $(*)$, it means an absolute constant. After defining the constants, we may abbreviate $N_i(*), C_i(*)$ as $N_i,C_i$. Given $\eta\in F$, we use $|\eta|<C$ to mean that for any real embedding $\iota:F\rightarrow \bR$, the absolute value $|\eta|_\iota<C$. We also use $|\cdot|$ to denote the absolute value on $\bC$.

\subsection{An upper bound of the values of Green function on Hecke orbits}
\begin{para}\label{cpx_setup}
Let $\cF\subset \bH^2$ be the fundamental domain for $\Gamma$ described in \cite{vdG}*{I.3} and $\overline{\cF}$ its closure (with respect to the complex analytic topology) in $\bH^2$ (that is, the cusps of $\Gamma\backslash \bH^2$ are not included). Let $\Omega\subset \overline{\cF}$ be a compact domain containing the preimage of $\bigcup_{r\in \bI} T(r)$ in $\Div(\Psi)$. Then there exists $C_0\in \bR_{>0}$ such that for any $(z_1,z_2)=(x_1+\sqrt{-1}y_1,x_2+\sqrt{-1}y_2)\in \Omega$, we have $|x_i|<C_0$ and $C^{-1}_0<y_i<C_0$.

For any $\overline{\bQ}$-point $[B]$ in $\cH$, we use $z(B)=(z_1(B),z_2(B))=(x_1(B)+\sqrt{-1}y_1(B),x_2(B),\sqrt{-1}y_2(B))$ to denote the preimage of $\sigma([B])$ in $\cF$.

Let $G$ be the pull back to $\bH^2$ of the Green function $-\log ||\Psi||_{\text{Pet}}$ of $\sum c_r T(r)$. There are only finitely many components of the preimage of $T(r)$ in $\Omega$ and for each component, we pick $(a,b,\gamma)$ such that $
\begin{pmatrix}
a &  \gamma\\
\gamma' & b
\end{pmatrix}
\in L$ with $ab-\gamma\gamma'=r\Nm \fa$ as in \cref{def_HZ} such that this component is defined by $az_1z_2+\gamma z_1+\gamma' z_2+b=0$. We use $\cM_{\Omega, r}$ to denote this finite set of $(a,b,\gamma)$.
Then by the definition of the Green function, we have that $G+\sum_{r\in \bI} c_r\sum_{(a,b,\gamma)\in \cM_{\Omega,r}}\log |az_1z_2+\gamma z_1+\gamma'z_2+b|$ is a real analytic function on $\overline{\cF}$.
\end{para}

The goal of this subsection is to show that for most $p$, one has that $$- \log |az_1(B)z_2(B)+\gamma z_1(B)+\gamma'z_2(B)+b|\leq O(\log p), \forall [B]\in T_{\fp}([A]).$$

\begin{proposition}\label{archBest}
Let $(a,b,\gamma)$ be a fixed triple in $\cM_{\Omega, r}$.
Given $C_1>23$ and $\epsilon_3>0$, there is an $N_0(\epsilon_3,C_1)>0$ such that for every $N>N_0(\epsilon_3,C_1)$, the number of primes in $[N^{1/2},N]$ for which there exists some $[B]\in T_{\fp}(\sigma[A])$ such that $$|az_1(B)z_2(B)+\gamma z_1(B)+\gamma'z_2(B)+b|< p^{-C_1}$$ is at most $\epsilon_3\#\{\text{primes}\in [N^{1/2},N]\}$.
\end{proposition}

We extend the idea in \cite{Ch} of relating bad primes and degrees of homomorphisms between well-chosen CM elliptic curves to the setting of Hilbert modular surfaces by using the theory of special endomorphisms.
A point $[B]$ on $\cH_{\bQ}$ is called \emph{special} if there exist $T(n_1)$ and $T(n_2)$, $n_1,n_2\in \bN$, $n_1n_2\notin (\bN)^2$ such that $[B]\in T(n_1)\cap T(n_2)$. If $[B]$ is special, then $B$ has complex multiplication.
We construct a special point $[A_\cm]$ on $\cH_{\bC}$ close to $\sigma([A])$ and show that if some point in $T_{\fp}(\sigma([A]))$ is close to $\Div(\Psi)=\sum_{r\in \bI} c_r T(r)$, then $A_\cm$ has a special endomorphism of certain degree. \Cref{archBest} then follows after analysis of the possible degree of special endomorphisms of $A_\cm$. In what follows, we will not specify the dependence of the constants $C_i$ in this subsection on the fixed triple $(a,b,\gamma)$.\\

The following lemma shows that if there exists $[B]\in T_{\fp}(\sigma([A]))$ which is close to $T(r)$, then $\sigma([A])$ is close to $T(pr)$.
\begin{lemma}\label{nearA}
If there exists $[B]\in T_{\fp}[A]$ such that $|az_1(B)z_2(B)+\lambda z_1(B)+\lambda'z_2(B)+b|<p^{-C_1}$, then there exist $m\in (D\Nm \fa)\bZ,l\in \bZ$ and $\eta\in \fa$ such that $ml-\Nm(\eta)=rp\Nm \fa$ and $|mz_1(A)z_2(A)+\eta z_1(A)+\eta'z_2(A)+l|<p^{-C_1}.$ Moreover, we have $|m|,|l|,|\eta|<C_5 p$. 
\end{lemma}

\begin{proof}
As in \cref{phecke}, we write $\fp=(\lambda)$ and after multiplying $\lambda$ by an element in $(\cO^\times)^2$, we may assume that $C^{-1}_6\sqrt{p}<|\lambda|<C_6\sqrt{p}$. We may also assume $z(B)\in \Omega$ (this can be done by letting $N$ be large enough). 

Let $U=
\begin{pmatrix}
u_{11} &  u_{12}\\
u_{21} & u_{22}
\end{pmatrix}
\in 
\begin{pmatrix}
1 & 0\\
0 & \lambda
\end{pmatrix}
\cdot\Gamma$ be the matrix that maps $z(A)$ to $z(B)$. The set $\Gamma$ acts on $V$ (in \cref{def_HZ}) via $g.M=(g')^tMg$ and this action preserves $L$. Let $
\begin{pmatrix}
m &  \eta\\
\eta' & l
\end{pmatrix}
=(U')^t
\begin{pmatrix}
a &  \gamma\\
\gamma' & b
\end{pmatrix}
U\in L$. Then $ml-\eta\eta'=\det(U')(ab-\gamma\gamma)\det(U)=rp\Nm \fa$ and 
$$mz_1(A)z_2(A)+\eta z_1(A)+\eta'z_2(A)+l=az_1(B)z_2(B)+\lambda z_1(B)+\lambda'z_2(B)+b.$$
This proves the first assertion. 

For the second assertion, we first bound $|u_{ij}|$. Consider the real embedding of $F$ corresponding to the first coordinate of $\bH^2$ and we will still use $u_{ij}$ to denote its image under this embedding. By definition of $U$, we have $\displaystyle z_1(B)=\frac{u_{11}z_1(A)+u_{12}}{u_{21}z_1(A)+u_{22}}$ and hence $\displaystyle y_1(B) =\frac {\lambda y_1(A)} {|u_{21}z_1(A)+u_{22}|^{2}}$. Since $y_1(B)>C_0^{-1}$ and $y_1(A)<C_0$, we have $|u_{21}z_1(A)+u_{22}|<C_0^2C_6\sqrt{p}$. Consider the imaginary part of $u_{21}z_1(A)+u_{22}$ and notice that $y_1(A)>C_0^{-1}$. Thus, we have $|u_{21}|<C_0^3C_6\sqrt{p}$. By considering the real part, we obtain
$$|u_{22}|\leq |u_{21}x_1(A)+u_{22}|+|u_{21}x_1(A)|\leq |u_{21}z_1(A)+u_{22}|+|u_{21}x_1(A)| \leq C_0^2C_6\sqrt{p}+C_0^4C_6\sqrt{p}.$$
On the other hand, using the bounds of $|u_{21}|,|u_{22}|$, we have $$|u_{11}|y_1(A)\leq |u_{11}z_1(A)+u_{12}|=|(z_1(B)(u_{21}z_1(A)+u_{22}))|\leq C_7 \sqrt{p}.$$ Hence we obtain
$$|u_{11}|\leq C_0C_7\sqrt{p},\, |u_{12}|\leq |u_{11}z_1(A)+u_{12}|+|u_{11}x_1(A)|\leq C_7\sqrt{p}+C_0^2C_7\sqrt{p}.$$
The same argument works for the other embedding of $F\rightarrow \bR$ by studying $z_2(B),z_2(A)$. The bounds of $|m|,|l|,|\eta|$ follow from the fact that $|u_{ij}|$ is bounded by $O(p^{1/2})$. 
\end{proof}

\begin{para}\label{quadratic}
The following lemma shows that if two Hecke orbits of $\sigma([A])$ satisfy the assumption of \Cref{nearA}, then $\sigma([A])$ is close to a special point on $\cH_{\bC}$. Recall that for a special point $[B]\in \cH(\bC)$, one defines a quadratic form $Q$, up to $\SL_2(\bZ)$-equivalence, as follows (see, for example, \cite{HZ}*{1.1}, \cite{vdG}*{V.4}). Let $L_{[B]}$ be the sub lattice of $L$ in \cref{def_HZ} such that for any $v\in L_{[B]}$, one has $(z_1(B)\quad 1)v(z_2(B)\quad 1)^t=0$. The lattice $L_{[B]}$ is of rank two and equipped with a natural orientation. The restriction of the quadratic form on $L$ to the rank two lattice $L_{[B]}$ is positive definite and it coincides, up to a constant, with the quadratic form on the $\bZ$-module of special endomorphisms of $B$.
By choosing an oriented basis, one obtains a positive definite binary integral quadratic form $Q$.
\end{para}

\begin{lemma}\label{archCM}
Assume that $N$ is large enough\footnote{In the proof, we give a constant $N_3(C_1)$ such that being large enough means $N>N_3$.} and that for primes $p_1,p_2 \in [N^{1/2},N], p_1<p_2$, there exist $[B]\in T_{\fp_1}[A],[B']\in T_{\fp_2}[A]$ satisfying $|az_1(B)z_2(B)+\lambda z_1(B)+\lambda'z_2(B)+b|<p_1^{-C_1}$ and $|az_1(B')z_2(B')+\lambda z_1(B')+\lambda'z_2(B')+b|<p_2^{-C_1}$. Then there exists a special point $[A_\cm]$ on $\cH_{\bC}$ such that $|z(A)-z(A_\cm)|<C_8N^{-C_1+7}$ and the integer coefficient binary quadratic form $Q_N$ associated to $[A_\cm]$ represents $p_1r$ and $p_2r$.
\end{lemma}

\begin{proof}
By \Cref{nearA}, there exist $m_i\in (D\Nm \fa)\bZ,l_i\in \bZ,\eta_i \in \fa$ with $|m_i|,|l_i|,|\eta_i|<C_5p_i$ such that $m_il_i-\Nm \eta_i=p_ir\Nm \fa$ and
$$|m_iz_1(A)z_2(A)+\eta_i z_1(A)+\eta_i'z_2(A)+l_i|<p_i^{-C_1}.$$
We first show that $|\eta'_i+m_iz_1(A)|$ is bounded below by a constant. Indeed, if $m_i\neq 0$, one has $|\eta'_i+m_iz_1(A)|\geq |m_i|y_1(A)\geq C_0^{-1}$ and if $m_i=0$, one has $|\eta'_i+m_iz_1(A)|=|\eta'_i|=\Nm \eta/|\eta'|=p_ir/|\eta'|\geq p_ir/|\eta_i|\geq r/C_5$. 

Therefore,
$$\left |z_2(A)+\frac{\eta_i z_1(A)+l_i}{\eta'_i+m_iz_1(A)}\right |\leq \frac{p_i^{-C_1}}{|\eta'_i+m_iz_1(A)|}\leq C_9 p_i^{-C_1}.$$
Let $f(z)$ be the $\cO$-coefficient quadratic polynomial in $z$ given by $(\eta_1z+l_1)(\eta'_2+m_2z)-(\eta_2z+l_2)(\eta'_1+m_1z)$. We first show that the leading coefficient $\eta_1m_2-\eta_2m_1\neq 0$. If not, then $f(z)=(m_2l_1-m_1l_2)z+l_1\eta_2'-l_2\eta_1'$ and $m_2l_1-m_1l_2\neq 0$ (otherwise, one would have $p_1=p_2$). In particular, $|m_2l_1-m_1l_2|\geq 1$. Hence $|f(z_1(A))|$ is bounded below by its imaginary part $|m_2l_1-m_1l_2|y_1(A)\geq C_0^{-1}$. On the other hand, since $|\eta'_i+m_iz_1(A)|\leq C_5p_i(1+2C_0)$, we have
$$|f(z_1(A))|\leq \big |(\eta'_1+m_1z_1(A))(\eta'_2+m_2z_1(A))\big |\cdot \left |\frac{\eta_1 z_1(A)+l_1}{\eta'_1+m_1z_1(A)}-\frac{\eta_2 z_1(A)+l_2}{\eta'_2+m_2z_1(A)}\right |\leq (C_5p_2+2C_0C_5p_2)^2\cdot (2C_9p_1^{-C_1}).$$
This leads to a contradiction, when $N^{\frac {C_1}{2}-2}>2C_0(C_5+2C_0C_5)^2C_9$. 

Now we show that $f$ has two complex roots. Let $\alpha,\beta$ be the two roots of $f$, then we have $f(z)=(\eta_1m_2-\eta_2m_1)(z-\alpha)(z-\beta)$. Then by assumption, we have
$$2C_9p_1^{-C_1}>\left |\frac{\eta_1 z_1(A)+l_1}{\eta'_1+m_1z_1(A)}-\frac{\eta_2 z_1(A)+l_2}{\eta'_2+m_2z_1(A)}\right |=\left |\frac{(\eta_1m_2-\eta_2m_1)(z_1(A)-\alpha)(z_1(A)-\beta)}{(\eta'_1+m_1z_1(A))(\eta'_2+m_2z_1(A))}\right |.$$
Since $m_i\in \bZ$ and $\eta_i\in \cO$, one has $|\eta_1m_2-\eta_2m_1|\geq |\eta_1'm_2-\eta_2'm_1|^{-1}\geq (2C_5^2p_1p_2)^{-1}$. Moreover, $|\eta'_i+m_iz_1(A)|\leq C_5p_i(1+2C_0)$. If $\alpha,\beta$ were real, then $|z_1(A)-\alpha|,|z_1(A)-\beta|\geq C_0^{-1}$ and one gets a contradiction when $N^{\frac{C_1}{2}-4}>4C_0^2C_5^4C_9(1+2C_0)^2$. Then we may assume that $\alpha$ has positive imaginary part.

We define $A_\cm$ to be the abelian surface such that $z_1(A_\cm)=\alpha$ and $\displaystyle z_2(A_{CM})=-\frac{\eta_1 z_1(A_{CM})+l_1}{\eta'_1+m_1z_1(A_{CM})}$. Then $A_{\cm}$ lies on both $T(p_1r)$ and $T(p_2r)$ (so $[A_{\cm}]$ is a special point). In other words, the integer coefficient binary quadratic form associated to $[A_\cm]$ represents $p_1r$ and $p_2r$.

Since $\beta$ has negative imaginary part, $|z_1(A)-\beta|\geq C_0^{-1}$. Since $p_2\leq \sqrt{p_1}$, the above discussion shows that 
$$|z_1(A)-z_1(A_{CM})|\leq 4C_0C_5^4C_9(1+2C_0)^2p_1^{-C_1+6}.$$
Moreover,
\begin{eqnarray*}
|z_2(A)-z_2(A_{CM})| & \leq  & \left |\frac{\eta_1 z_1(A)+l_1}{\eta'_1+m_1z_1(A)}-\frac{\eta_1 z_1(A_{CM})+l_1}{\eta'_1+m_1z_1(A_{CM})} \right |+\left |z_2(A)+\frac{\eta_i z_1(A)+l_i}{\eta'_i+m_iz_1(A)}\right | \\
& \leq & \frac {(l_1m_1-\eta_1\eta_1')|z_1(A)-z_1(A_{CM})|}{|(\eta'_1+m_1z_1(A))(\eta'_1+m_1z_1(A_{CM}))|}+C_9p_1^{-C_1}\\
& \leq &  4C_0C_5^4C_9(1+2C_0)^2(\max\{C_0,C_5/r\})C_{10}rp_1^{-D+7},
\end{eqnarray*}
where $C_{10}^{-1}$ is the lower bound of $y_1(A_\cm)$ given by $\min\{C_0^{-1}-4C_0C_5^4C_9(1+2C_0)^2p_1^{-C_1+6},r/C_5\})$.
Hence we take $C_8$ large enough and conclude that $z(A)$ is close to $z(A_\cm)$.
\end{proof}

The following lemma shows that the special point $[A_\cm]$ constructed above is unique when $N$ is large enough.
\begin{lemma}\label{archCMunique}
Assume that $N$ is large enough\footnote{In the proof, we give a constant $N_4(C_1)$ such that being large enough means $N>N_4$.} and that for $p_i\in [N^{1/2},N]$ with $i=1,2,3,4$, $p_1\neq p_2$, $p_3\neq p_4$, there exists $[B_i]\in T_{{\fp}_i}([A])$ such that $|az_1(B_i)z_2(B_i)+\lambda z_1(B_i)+\lambda'z_2(B_i)+b|<p_i^{-C_1}$.
Let $[A_1],[A_2]$ be two special points constructed as in \Cref{archCM} by using the assumption on $p_1,p_2$ and $p_3,p_4$. Then $[A_1]=[A_2]$. More precisely, if $|z(A_1)-z(A_2)|\leq 2C_8N^{-C_1+7}$, then $z(A_1)=z(A_2)$.
\end{lemma}

\begin{proof}
Let $f_1(z),f_2(z)\in \cO[z]$ be the quadratic equations defining $z_1(A_1),z_1(A_2)$ in the proof of \Cref{archCM}. Let $\alpha_1,\alpha_2\in F$ be the leading coefficients of $f_1,f_2$ and $-\Delta_1,-\Delta_2$ the discriminants. Since $f_i$ has two complex roots for both embeddings of $F$, one has that $\Delta_i$ is totally real. Given a real embedding of $F$, we may assume both $\alpha_1,\alpha_2$ are positive with respect to this embedding. 

By the definition of $f_i$ and \Cref{nearA}, we have $|\alpha_i|\leq 2C_5^2N^2$ and $|\Delta_i|\leq C_{11}N^4$. Moreover, since $\alpha_i,\Delta_i\in \cO$, we have the nonzero $|\Nm(\alpha_2^2\Delta_1-\alpha_1^2\Delta_2)|\geq 1$ and hence $$|\alpha_2^2\Delta_1-\alpha_1^2\Delta_2|\geq |(\alpha'_2)^2\Delta'_1-(\alpha'_1)^2\Delta'_2|^{-1}\geq (8C^2_5C_{11})^{-1}N^{-8}.$$
Putting these inequalities together, we obtain
\begin{eqnarray*}
|z(A_1)-z(A_2)| & \geq & |z_1(A_1)-z_1(A_2)|\geq |y_1(A_1)-y_1(A_2)|=\Big|\frac{\sqrt{\Delta_1}}{2\alpha_1}-\frac{\sqrt{\Delta_2}}{2\alpha_2} \Big|\\
&= & \frac{|\alpha_2^2\Delta_1-\alpha_1^2\Delta_2|}{2\alpha_1\alpha_2(\alpha_1\sqrt{\Delta_2}+\alpha_2\sqrt{\Delta_1})} \geq C_{12}N^{-16}.
\end{eqnarray*}
This contradicts our assumption when $N^{C_1-23}>2C_8C_{12}^{-1}$.
\end{proof}

\begin{corollary}\label{archrep}
Assume that $N$ is large enough as above and that $A$ satisfies the assumption in \Cref{archCM}. For any $p_3\in [\sqrt{N},N]$ such that there exists $[B'']\in T_{\fp_3}[A]$ satisfying $|az_1(B'')z_2(B'')+\lambda z_1(B'')+\lambda'z_2(B'')+b|<p_3^{-C_1}$, the quadratic form $Q_N$ in \Cref{archCM} represents $p_3r$.
\end{corollary}

\begin{proof}
By \Cref{archCM}, we construct special points $[A_1]$ by using $p_1,p_2$ and $[A_2]$ by using $p_1,p_3$. By \Cref{archCMunique}, we have $[A_1]=[A_2]$ and hence they have the same quadratic form $Q_N$. Since $[A_2]$ lies on $T(p_3r)$, then $Q_N$ represents $p_3r$.
\end{proof}

\begin{lemma}\label{DiscInfty}
Fix $C_1>23$ and let $\Delta_N$ denote the discriminant of $Q_N$ in \Cref{archCM}. As $N\rightarrow \infty$, we have $|\Delta_N|\rightarrow \infty$. 
\end{lemma}

\begin{proof}
Fix a bound $X$ of $|\Delta_N|$, then there are only finitely many equivalent classes of integral binary quadratic forms of discriminant $\leq X$. For each class, \cite{HZ}*{Thm.~1}\footnote{Although in \cite{HZ}, they assume that $D$ is a prime, their method still works in general. See for example \cite{vdG}*{V.6}.} shows that are only finitely many special points corresponding to the given class of quadratic forms. As $N\rightarrow \infty$, the CM approximation $[A_{\cm,N}]$ is closer to $\sigma[A]$ and hence $|\Delta_N|$ cannot be bounded.
\end{proof}

\begin{proof}[Proof of \Cref{archBest}]
Let $p_1,p_2$ be the smallest primes in $[\sqrt{N},N]$ such that there exists $B\in T_{\fp_i}A$ such that $|az_1(B)z_2(B)+\lambda z_1(B)+\lambda'z_2(B)+b|<p_i^{-C_1}$. Then by \Cref{archCM} and \Cref{archrep}, we obtain a quadratic form $Q_N$ associated to a special point $[A_{\cm,N}]$ which represents $p_3r$ for any prime $p_3$ in $[\sqrt{N},N]$ which satisfies the condition that there exists $[B'']\in T_{\fp_3}[A]$ satisfying $|az_1(B'')z_2(B'')+\lambda z_1(B'')+\lambda'z_2(B'')+b|<p_3^{-C_1}$. There exists a positive definite integral binary quadratic form $Q_N'$ such that a prime $p$ is represented by $Q_N'$ if and only if $pr$ is represented by $Q_N$. One also has that the absolute value of the discriminant $\Delta_N'$ of $Q_N'$ is at least $|\Delta_N|/r^2$.

It remains to show that the density $\displaystyle \frac{\{\text{prime }p \in [N^{1/2},N], p \textrm{ is represented by } Q'_{N}\}}{\{p \in [N^{1/2},N]\}}\rightarrow 0$ as $N\rightarrow \infty$. 
When $|\Delta_N'| \leq (\log N)^4$, \cite[Corollary 1.3]{TZ} shows that 
$\displaystyle \{\text{prime } p \in [N^{1/2},N], p \textrm{ is represented by } Q_{N}\} \ll \frac{\Li(N)}{h_N}$, where $h_N$ is the number of $\SL_2(\bZ)$-equivalence classes of primitive positive definite integral binary quadratic forms of discriminant $\Delta_N'$. Since $\Delta_N'\rightarrow -\infty$ by \Cref{DiscInfty}, one has $h_N\rightarrow \infty$.
When $|\Delta_N'|>(\log N)^4$, \cite[Lemma 5.2]{Ch} shows that $$ \{\text{integer } n\in[N^{1/2},N], n \text{ is represented by }Q_N'\}\leq 1+4\sqrt{2N}+8N/|\Delta_N'|^{1/2}=O\left(\frac{N}{(\log N)^2}\right). $$ 
We get the desired property by putting these two cases together.
\end{proof}

\subsection{From equidistribution to an upper bound of archimedean contribution}
This is the main theorem of this section. We use \Cref{archBest} and equidistribution theorem for Hecke orbits to show that for most $p$, the archimedean contribution in the height of $T_{\fp}([\cA])$ is $o(p\log p)$.
\begin{theorem}\label{thm_arch}
For any $\epsilon_1, \epsilon_2>0$, there is an $N(\epsilon_1,\epsilon_2)>0$ such that for every $N>N(\epsilon_1,\epsilon_2)$, the number of primes in the interval $[N^{1/2},N]$ for which $$-\sum_{[B]\in T_{\fp}[A]}\log ||\Psi(\sigma([B]))||_\pet\geq \epsilon_1 p\log p$$ is at most $\epsilon_2\#\{\ell \in [N^{1/2},N] \text{ prime}\}$.
\end{theorem}

\begin{proof}
Notation as in \cref{cpx_setup}.
We first show that a fixed triple $(a,b,\gamma)$, 
$$-\sum_{B\in T_{\fp} A}\log |az_1(B)z_2(B)+\gamma z_1(B)+\gamma'z_2(B)+b|=o(p\log p).$$

For any $\epsilon>0$, by the equidistribution theorem of Hecke orbits (see for example \cite{COU}), there exist constants $N_1(\epsilon, C_1)>0$ and $C_3(\epsilon, C_1)<0$ such that for any $p>N_1$, 
$$\#\{[B]\in T_{\fp}[A] : \log |az_1(B)z_2(B)+\gamma z_1(B)+\gamma'z_2(B)+b|<C_3\}<\epsilon p/C_1.$$

Let $\bI'=\{r\in \bI \mid c(r)>0\}$ and $M=\sum_{r\in \bI'}\# \cM_{\Omega,r}$. Taking $\epsilon_3=\epsilon_2/M$ and applying \Cref{archBest} for all triples $(a,b,\gamma)\in \bigcup_{r:c(r)>0}\cM_{\Omega,r}$, we have that, for every $N>N_2(\epsilon_2,C_1)\colonequals \max_{(a,b,\gamma)}\{N_0(a,b,\gamma, \epsilon_3, C_1),N_1\}$,
$$-\sum_{[B]\in T_{\fp} [A]}\log |az_1(B)z_2(B)+\gamma z_1(B)+\gamma'z_2(B)+b|<-(p+1)C_3+(\epsilon p/C_1)\cdot C_1\log p$$
holds for primes $p\in [N^{1/2},N]$ outside a set $\bB_N$ of density $\epsilon_2$ (this set is the union of the exceptional sets for all $(a,b,\gamma)\in \cup_{r:c(r)>0}\cM_{\Omega,r}$).

Let $\phi$ be a smooth function which is $1$ in $\Omega$ with compact support in $\overline{\cF}$.
Then by \cref{cpx_setup}, the function $f=G+\sum_{r\in \bI} c(r) \sum_{(a,b,\gamma)\in \cM_{\Omega,r}} \phi(z)\log |az_1z_2+\gamma z_1+\gamma'z_2+b|$ is smooth on $\overline{\cF}$. Since $G$ and hence $f$ go to $-\infty$ as $y_1y_2$ goes to $\infty$, we see that $f$ is bounded above on $\overline{\cF}$. On the other hand, since $\phi(z)$ has compact support, $\phi(z)\log |az_1z_2+\gamma z_1+\gamma'z_2+b|$ is also bounded above. Therefore, since $\#T_\fp[A]=p+1$,
 we have $$\sum_{[B]\in T_{\fp}[A]}\big (G(z(B))+\sum_{r\in \bI'}c(r)\sum_{(a,b,\gamma)\in\cM_{\Omega,r}} \phi(z)\log |az_1(B)z_2(B)+\gamma z_1(B)+\gamma'z_2(B)+b|\big )<C_2\cdot (p+1).$$

Take $\epsilon=\epsilon_1/(2M)$, then we have, for $N>N_2$, for $p\in [N^{1/2},N]\backslash \bB_N$  ,
$$\sum_{[B]\in T_{\fp}[A]}G(z(B))<C_{4}(p+1)+(\epsilon_1p\log p)/2.$$
Then by taking $N(\epsilon_1,\epsilon_2)>N_2$ large enough so that $C_{4}N(\epsilon_1,\epsilon_2)^{1/2}<(\epsilon_1 N(\epsilon_1,\epsilon_2)^{1/2}\log N(\epsilon_1,\epsilon_2))/4$, the theorem follows.
\end{proof}

\section{Special endomorphisms and contributions at finite places}\label{finite}
In this section, we will bound the local intersection multiplicities of $(T_\fp([\cA]),\cT(r))$ at non-archimedean places for $r\in \bI$, where $\bI$ is a fixed finite set such that $D|r$ and $\cT(r)$ is compact in $\cH$ for all $r\in \bI$. The set $\bI$ will be chosen by \Cref{cpt_div}. Throughout, $\ell$ denotes a prime, and $B,B',\cB,\cB'$ denote abelian surfaces with $\cO_F$-multiplication and $\fa$-polarization. We fix a finite place $v$ of a number field $K$ over $\ell$. Recall that $e_v$ is the ramification degree of $K$ at $v$. The abelian surfaces may be defined over $K,\cO_K$, $\cO_{K_v}$, or $\cO_{K_v}/v^n$. Recall that we use $\cB,\cB'$ to denote abelian surfaces defined over $\cO_{K_v}$ and we use $\cB_{v,n},\cB'_{v,n}$ to denote their reduction modulo $v^n$. We will use $\cB_{\ell,n}$ and $\cB'_{\ell,n}$ to denote surfaces over $\cO_{K_v}/\ell^n$ (note that the abelian varieties are defined modulo $\ell^n$, not $v^n$) which do not \emph{a priori} come with lifts to $\cO_{K_v}$. 
 Let $M_{v,n}$ and $M_{\ell,n}$ denote the module of special endomorphisms of $\cB_{v,n}$ and $\cB_{\ell,n}$ respectively. Finally, we let $\Lambda_v = \End(\cB[\ell^{\infty}]) \cap M_{v,1} \otimes \bZ_{\ell} $, where the intersection takes place in $\End(\cB_{v,1}) \otimes \bZ_{\ell} = \End(\cB_{v,1}[\ell^{\infty}])$. We call the $\bZ_{\ell}$-module $\Lambda_v$ the set of special endomorphisms of $\cB[\ell^{\infty}]$

\subsection{Deformation theory}
%$\# \{s \in \End(\cB_{v,n}): s\textrm{ is special}, Q(s) \leq N\} = O(\frac{N^{\frac{m}{2}}}{\ell^{mn/e_v}} + \frac{N^{\frac{m-1}{2}}}{\ell^{(m-1)n/e_v}} + \hdots \frac{N^{\frac{m'+1}{2}}}{\ell^{(m'+1)n/e_v}} + \frac{N^{\frac{m'}{2}}}{D_n} + N^{\frac{m'-1}{2}})$.

%\footnote{In later applications, we will choose $n$ depending on $N$, but the choice of $n$ always satisfies the condition that $\ell^{2mn/e_v}\ll N^m$; otherwise, the theorem still holds if we replace $O(N^m/\ell^{2mn/e_v})$ by $O(N^m/\ell^{2mn/e_v})+O(1)$.}  

%\footnote{In later applications, we will choose $M$ depending on $N$, but the choice always satisfies the condition that $M^{2m}\ll N^m$; otherwise, in the statement of Davenport's Lemma, we replace $O(N^{m})/M^{2m}$ by $O(N^{m})/M^{2m}+O(1)$.} 

The following result is crucial to bounding the local intersections: 

\begin{theorem}\label{numspecial}
 Let $m$ be the $\bZ$-rank of $M_{v,1}$ and $m'$ be the $\bZ_{\ell}$-rank of $\Lambda_v$.  Then there exists a positive integer $n_0$ such that $M_{n'_0 + ke_v,v} = (\Lambda_v + \ell^kM_{v,n'_0}\otimes \bZ_{\ell}) \cap M_{n_0,v}$ where $n'_0 \geq n_0$ and $k$ is allowed to be any positive integer.
\end{theorem}

In an earlier version of the paper, we had in an earlier version of this paper implicitly assumed that $\Lambda_v = 0$ (this assumption simplified the geometry-of-numbers arguments), and we are very grateful to Davesh Maulik for pointing this out to us.  We will need the following result on homomorphisms between abelian surfaces:
\begin{lemma}\label{GMes}
Let $\alpha \in \Hom(\cB_{\ell,n-1},\cB'_{\ell,n-1})$ for any $n \geq 3$. 
\begin{enumerate}
\item
The homomorphism $\ell\alpha$ lifts uniquely to $\Hom(\cB_{\ell,n},\cB'_{\ell,n})$.
\item
If $\ell \alpha$ lifts to $\Hom(\cB_{\ell,n+1},\cB'_{\ell,n+1})$, then $\alpha$ lifts to $\Hom(\cB_{\ell,n},\cB'_{\ell,n})$.

\end{enumerate}
\end{lemma}
Let $\pdiv_i = \cB_{\ell,i}[\ell^{\infty}]$, and $\pdiv'_i = \cB'_{\ell,i}[\ell^{\infty}]$. By the Serre--Tate lifting theorem, it suffices to prove the analogous result for $\ell$-divisible groups. This is a straightforward application of Grothendieck--Messing theory. Before proceeding to the proof, we recall some facts from Grothendieck--Messing theory (\cite{Messing} contains every result that we need). All the reduction maps between the $\cO_{K_v}/\ell^{i}$ for $i = n, n\pm1$ are canonically equipped with nilpotent divided powers (in fact, as $n > 2$, all the ideals in play are square-zero). Let $\bD$ and $\bD'$ denote the Dieudonne-crystals associated to $\pdiv_{n-1}$ and $\pdiv'_{n-1}$ (see \cite[\S 2.5 of Chapter IV]{Messing}). Any homomorphism between $\pdiv_{n-1}$ and $\pdiv'_{n-1}$ canonically induces a map of crystals $\bD \rightarrow \bD'$. 

Let $D_i$ and $D'_i$ denote $\bD$ and $\bD'$ evaluated at $\cO_{K_v}/\ell^i$ for $i = n, n\pm1$ (these are free $\cO_{K_v}/\ell^i$-modules whose ranks equal the heights of $\pdiv$ and $\pdiv'$). Grothendieck--Messing theory associates canonical filtrations $F_i \subset D_i$ and $F'_i \subset D'_i$ to the groups $\pdiv_i$ and $\pdiv'_i$ for $i = n, n\pm1$. Note that the $F_{n+1}$ reduces to $F_n$ and $F_{n-1}$ under the canonical quotient maps (the analogous statement holds for $F'_{n+1}$).  The filtrations are direct summands of the crystals evaluated at the $\cO_{K_v}/\ell^i$. Suppose that $W_{n+1} \subset D'_{n+1}$ is some submodule such that $D'_{n+1} = F'_{n+1} \oplus W_{n+1}$. Let $W_n \subset D'_n$ denote the mod-$\ell^n$ reduction of $W_{n+1}$. Clearly, $W_n \oplus F'_n = D'_n$. By \cite[Theorem 1.6 of Chapter V]{Messing}, a homomorphism between $\pdiv_{n-1}$ and $\pdiv'_{n-1}$ lifts to $\Hom(\pdiv_i,\pdiv'_i)$ (for $i = n,n+1$) if and only if the associated map of crystals evaluated at $\cO_{K_v}/\ell^i$ maps the filtration $F_i$ to $F'_i$. Let $\alpha_i: D_i \rightarrow D'_i$ ($i = n, n\pm1$) denote the maps induced by $\alpha$. We now proceed to the proofs of the two statements. 

\begin{proof}[Proof of \Cref{GMes}]

\ \begin{enumerate}
\item
Let $v_n \in F_n \subset D_n$, whose image in $F_{n-1}$ is denoted by $v_{n-1}$. It suffices to prove that $\ell \alpha_n(v) \in F'_n$. Let $\alpha_n(v) = v_n + w_n$, where $v'_n \in F'_n$ and $w_n \in W_n$. As $\alpha_{n-1}(F_{n-1}) \subset F'_{n-1}$, we have $w_n$ modulo $\ell^{n-1}$ is zero. It follows that $\ell w_n = 0$. Therefore, $\ell \alpha_n(v_n) = \ell v'_n \in F'_n$. Thus $\ell \alpha_n$ preserves filtrations, as requried. 

\item
As above, let $v_n \in F_n$. Let $v_{n+1} \in F_{n+1}$, whose mod-$\ell^n$ reduction is $v_n$. Suppose that $\alpha_{n+1}(v_{n+1}) = v'_{n+1} + w_{n+1}$, where $v'_{n+1} \in F'_{n+1}$ and $w_{n+1} \in W'_{n+1}$. As $\ell \alpha$ lifts to $\Hom(\pdiv_{n+1}, \pdiv'_{n+1})$, it follows that $\ell w_{n+1} = 0$. Therefore, $w_{n+1} =0$ modulo $\ell^n$. It follows that $\alpha_{n+1}(v_{n+1})$ modulo $\ell^n$ - which equals $\alpha_n(v_n)$ - is an element of $F'_n$. 

\end{enumerate}
\end{proof}

We now prove \Cref{numspecial} 
\begin{proof}[Proof of \Cref{numspecial}]
For ease of notation, denote by $\Lambda_{v,i}$ the $\bZ_{\ell}$-module $M_{v,i} \otimes \bZ_{\ell}$. 
By the Serre--Tate theorem, it suffices to prove the existence of $n_0$ such that $\Lambda_{v,n'_0 + ke_v}= \Lambda_v + \ell^k\Lambda_{v,n'_0}$. First, note that \Cref{GMes} implies that $\Lambda_v \subset \Lambda_{v,2e_v}$ is co-torsion free. Let $\Lambda' \subset \Lambda_{v,2e_v}$ denote a direct summand of $\Lambda_v$. As the $\bZ_\ell$-module of special endomorphisms of $\cB[\ell^{\infty}] = \Lambda_v$, it follows that  Therefore, $\bigcap_n (\Lambda' \cap \Lambda_{v,n}) = 0$. 

The theorem follows directly from the following claim.
\begin{claim}
We have that $\Lambda' \cap \Lambda_{v,n + e_v}\subset \Lambda' \cap \ell \Lambda_{v,n} $ for large enough $n$.
\end{claim}
To prove the claim, we fix any $n' > 2e_v$. Since $\bigcap_n (\Lambda' \cap \Lambda_{v,n}) = 0$, then $\Lambda' \cap \Lambda_{v,n' + ke_v}\subset \ell (\Lambda' \cap \Lambda_{v,n'})$ for large enough $k$. 
%Let $\alpha \in M_{v,n_0 + ke_v}$ denote some element which is not a multiple of $\ell$. If $\alpha$ lifts to $M_{v,n_0 + (k+1)e_v}$, then by \Cref{GMes}, $\alpha$ thought of as an element of $M_{v,n_0 + (k-1)e_v}$ is not be a multiple of $\ell$. By iterating this argument, it follows that $\alpha \in M_{v,n_0}$ is not a multiple of $\ell$. This contradicts our assumption that $M_{v,n_0 + ke_v} \subset \ell M_{v,n_0}$.
We now prove by contradiction that $\Lambda' \cap \Lambda_{v,n'+(k+1)e_v}\subset \ell(\Lambda' \cap \Lambda_{v,n'+ke_v})$ for such $k$. Assume that there exists a special endomorphism $\alpha \in (\Lambda' \cap \Lambda_{v,n'+(k+1)e_v})\backslash \ell(\Lambda' \cap \Lambda_{v,n'+ke_v})$. If $\alpha \in \ell (\Lambda' \cap \Lambda_{v,n'+(k-1)e_v})$, we write $\alpha=\ell \beta$, where $\beta \in \Lambda' \cap \Lambda_{v,n'+(k-1)e_v}$. By assumption, $\ell\beta\in \Lambda' \cap \Lambda_{v,n'+(k+1)e_v}$ and then by \Cref{GMes}, $\beta\in \Lambda' \cap \Lambda_{v,n'+ke_v}$. This contradicts that $\alpha\notin \ell(\Lambda' \cap \Lambda_{v,n'+ke_v})$ and hence we have shown that $\alpha \notin \ell(\Lambda' \cap \Lambda_{v,n'+(k-1)e_v})$. By iterating this argument, it follows that $\alpha \notin \ell(\Lambda' \cap \Lambda_{v,n'})$, which is a contradiction.
\end{proof}

\subsection{Geometry of numbers and applications to counting special endomorphisms}
For an $m$-dimensional lattice $M$ with a positive definite quadratic form $Q$, let $\mu_1(M) \leq \mu_2(M) \hdots \leq \mu_m(M)$ denote the successive minima of $M$ (see \cite[Definition 2.2]{Esk} for the definition of the term successive minima). We will need the following lemma due to Schmidt: 

\begin{lemma}\label{dav}
Then $\#\{s \in M \mid Q(s) \leq N \} = O\left(\displaystyle{\sum_{j = 0}^m \frac{N^{j/2}}{\mu_1(M) \hdots \mu_{j}(M)}}\right)$, where the implied constant depends only on $m$.
\end{lemma}
\begin{proof}
Equations (5) and (6) on page 518 of \cite{Esk} imply that Lemma 2.4 of \emph{loc. cited} implies the stated result (the authors refer to \cite{Sch} for a proof of Lemma 2.4). 
\end{proof}
We now prove a elementary lemma that will allow us to use \Cref{dav} to bound special endomorphisms of $B_{v,n}$. We refer to the beginning of \S \ref{finite} for notation. 
\begin{lemma}\label{longshortvector}
Let $m$ be the $\bZ$-rank of $M_{v,1}$ and $m'$ be the $\bZ_{\ell}$-rank of $\Lambda_v$. Then $\displaystyle{\prod_{i=1}^j \mu_i(M_{v,n})} \gg \ell^{n(j - m')/e_v}$.
\end{lemma}
\begin{proof}
We may assume that $m' <j$. It suffices to prove that $\displaystyle{\prod_{i=1}^j \mu_i(M_{v,n_0 + (k+1)e_v}) \geq \ell^{j - m'}\prod_{i=1}^j \mu_i(M_{v,n_0 + ke_v})}$. 

For a lattice $M$, let $d(M)$ denote the square root of its discriminant. \Cref{numspecial} implies that $d(M_{n_0 + (k+1)e_v }) \geq \ell^{m-m'} d(M_{n_0 + ke_v})$. Thus, 
\begin{equation}\label{j=m}
\displaystyle{\prod_{i=1}^m \mu_i(M_{v,n_0 + (k+1)e_v}) \geq \ell^{m'}\prod_{i=1}^m \mu_i(M_{v,n_0 + ke_v})}
\end{equation} by \cite{Esk}*{Eqn.~(5),(6)}. This is the desired result for $j=m$. 

Moreover, if $M \subset M'$ are lattices, then $\mu_i(M) \geq \mu_i(M')$. Therefore, \Cref{numspecial} implies that $\mu_i(M_{n_0 + (k+1)v}) \leq \mu_i(\ell M_{n_0 + kv})=\ell \mu_i(M_{n_0 + kv})$. The lemma follows from multiplying (\ref{j=m}) with the inequality $\displaystyle{\prod_{i=j+1}^m\mu_i(M_{v,n_0 + (k+1)e_v})^{-1} \geq \prod_{i=j+1}^m \ell^{-1}\mu_i(M_{v,n_0 + ke_v})^{-1}}$. 
\end{proof}

\Cref{dav} and \Cref{longshortvector} immediately yield the following corollary: 
\begin{corollary}\label{onedim}
Suppose that the $\bZ_{\ell}$-rank of \revise{$\Lambda_v$} is $\leq 1$ and the rank of $M_{v,1}$ is $m$. Then $$\# \{s\in M_{v,n}\mid Q(s) \leq N\} = O\Big(N^{1/2}+ \displaystyle{\sum_{j = 2}^{m} \frac{N^{j/2}}{\ell^{(j-1)n/e_v}}} \Big) $$.
\end{corollary}

\subsection{Proof of the non-archimidean local results}
In what follows, we consider $T_{\fp}$ as in \cref{phecke} with $p\neq \ell_v$ and recall that we write $(p)=\fp \fp'\subset \cO_F$. We always use $N$ to denote a large enough integer.
\begin{lemma}\label{hecktate}
Over $\bZ[1/pr]$, we have $T_{\fp}\cT(r)=\cT(pr)=T_{\fp'}\cT(r)$. Moreover, for any $n$, if there exists $[\cB]\in T_{\fp}[\cA]$ such that $[\cB_{v,n}] \in \cT(r)$, then $[\cA_{v,n}]\in \cT(pr)$.
\end{lemma}
\begin{proof}
By checking on complex points, we have $T_{\fp}T(r)=T(pr)=T_{\fp'}\cT(r)$. By definition, $\cT(m)$ is the Zariski closure of $T(m)$ and hence $T_{\fp}\cT(r)=\cT(pr)$. The second assertion then follows from the \'etaleness of Hecke orbits. 
\end{proof}

For the rest of this section, we assume $\End(A_{\overline{\bQ}})=\cO_F$ and $\cA$ has good reduction at $v$. Further, the \emph{norm} of a special endomorphism $s$ denotes the integer $Q(s)$. 

The following lemma is well-known and follows directly from the crystalline realization of the module of special endomorphisms. We record a proof here for completeness.
\begin{lemma}\label{twodim}
Let $m$ be the $\bZ$-rank of $M_{v,1}$ and $m'$ be the $\bZ_{\ell}$-rank of $\Lambda_v$. Then $m\leq 4$ and $m'\leq 2$.
\end{lemma}
\begin{proof}
Since $s^*=s$ for any $s\in M_{v,1}$ (resp. $\Lambda_v$), we have that $\Lambda_v\subset M_{v,1}\otimes \bZ_\ell \subset H^2_{\cris}(\cB_{v,1}/W(\bar{\bF}_\ell))$. On the other hand, since $\cO_F\subset \End(\cB_{v,1})$ is stable under Rosati involution, we have a natural embedding $\cO_F\subset H^2_{\cris}(\cB_{v,1}/W(\bar{\bF}_\ell))$.
Since $H^2_{\cris}(\cB_{v,1}/W(\bar{\bF}_\ell))[1/\ell]$ is a $W(\bar{\bF}_\ell))[1/\ell]$-vector space of dimension $6$, then the Frobenius invariant part of $H^2_{\cris}(\cB_{v,1}/W(\bar{\bF}_\ell))[1/\ell]^{\varphi=1}$ is a $\bQ_\ell$-vector space of dimension at most $6$.

Since $s\circ f=f'\circ s$ for any $s\in M_{v,1}\otimes \bQ_\ell$ and $f\in F$, then $\cO_F\otimes \bQ_\ell \cap M_{v,1}\otimes \bQ_\ell=0$ in $H^2_{\cris}(\cB_{v,1}/W(\bar{\bF}_\ell))[1/\ell]^{\varphi=1}$. Since $\cO_F\otimes \bQ_\ell$ has dimension $2$, we have that $M_{v,1}\otimes \bQ_\ell$ is at most of dimension $4$. Since $M_{v,1}\inj M_{v,1}\otimes \bQ_\ell$, we have that $m\leq 4$.

On the other hand, the de Rham cohomology of $\cB$ induces a (decreasing) Hodge filtration $\Fil^\bullet$ on $H^2_{\cris}(\cB_{v,1}/W(\bar{\bF}_\ell))\otimes \bar{\bQ}_\ell$ with $\dim \Fil^0=5$ and $\dim \Fil^1=1$. Hence $\Fil^0\cap H^2_{\cris}(\cB_{v,1}/W(\bar{\bF}_\ell))[1/\ell]^{\varphi=1}$ is a $\bQ_\ell$-vector space of dimension at most $5$.
By Grothendieck--Messing theory, both $\cO_F$ and $\Lambda_v$ lie in $\Fil^0$ and hence $m'+2\leq 5$. If $m'=3$, then $\Span\{\cO_F, \Lambda_v\}=\Fil^0$. By Mazur's weak admissibility theorem, since both $\cO_F$ and $\Lambda_v$ lie in $H^2_{\cris}(\cB_{v,1}/W(\bar{\bF}_\ell))[1/\ell]^{\varphi=1}$, then $\Span\{\cO_F, \Lambda_v\}$ only has trivial filtration. This contradicts that $0\neq \Fil^1\subset \Fil^0$ and we conclude that $m'\leq 2$.
\end{proof}

\begin{theorem}\label{equinon}
Let $M(N,n,\epsilon)$ denote the number of primes $p \in [N^{1/2},N]$ such that $\# \{T_{\fp}([\cA_{v,n}]) \cap (\bigcup_{r\in \bI}\cT(r))\} \geq \epsilon p$. Then $M(N, \lceil 3e_v\log \log N \rceil,\epsilon) = o(N/(\log N))$.
\end{theorem}
\begin{proof}
The number of primes in the interval $[N^{1/2},N/\log(N)]$ is $o(N/\log(N)$, so we will restrict ourselves to primes $p \in [N/\log(N),N]$. For each prime $p$, each $[\cB_{v,n}] \in T_{\fp}([\cA_{v,n}]) \cap (\cup_{r\in \bI}\cT(r))$ induces a special endomorphism of $\cA_{v,n}$ whose norm is $pr\Nm \fa/D$. For all $p\in [N/\log N,N]$, the quantity $pr\Nm \fa/D=O(N)$. Notice that distinct $[\cB_{v,n}]\in T_{\fp}([\cA_{v,n}]) \cap (\cup_{r\in \bI}\cT(r))$ induce distinct special endomorphisms of $\cA_{v,n}$. Therefore, $\cA_{v,n}$ has at least $M(N,n,\epsilon) \epsilon N/\log N$ special endomorphisms with norm bounded by $N$. 

Applying the crudest bounds that \Cref{numspecial}, \Cref{dav} and \Cref{twodim} yield, the number of special endomorphisms $\cA_{v,n}$ has with norm bounded by $N$ is $O\bigg(\frac{N^2}{\ell^{2n/e_v}} + N^{3/2}\bigg)$. Therefore $M(N,n,\epsilon) = O\bigg(\frac{N \log N}{ \ell^{2n/e_v}} + N^{1/2}\log N\bigg)$ ($\epsilon$ gets absorbed in the $O()$). Substituting $n = \lceil 3e_v\log \log N \rceil$ yields $M(N,n,\epsilon)=o(N/(\log N))$ as required. 
\end{proof}

The following theorem shows that one can choose a sequence of $p$ such that the largest $v$-adic intersection multiplicity of a point in $T_\fp([\cA])$ with $\bigcup_{r\in \bI} \cT(r)$ is $O(\log p)$. 

\begin{theorem}\label{bestnonarch}
Set $n = \lceil \frac{e_v \log N}{ \log \ell} \rceil$. Then the number of primes $p \in [N^{1/2},N]$ for which there exists $[\cB] \in T_\fp([\cA])$ and $r\in \bI$ with $[\cB_{v,n}] \in \cT(r)$ is $o(N/\log N)$.

\end{theorem}
%\yunqing{I think the proof of main theorem does not require change since I did not use the sharpest bound previously. Please double check!}
\begin{proof}
Let $p \in [N^{1/2},N]$ such that there exists $[\cB] \in T_\fp([\cA])$ and $r \in \bI$ as in the statement. Then $\cA_{v,n}$ has a special endomorphism, say $s_p$, of norm $pr\Nm \fa/D$. Clearly, $s_p \neq s_{p'}$ where $p' \neq p$ also satisfies the conditions in the statement. Therefore, each such $p$ induces a distinct special endomorphism of $\cA_{v,n}$ having norm $O(N)$ and it suffices to bound the number of special endomorphisms of norm $\leq N$. 

Let $\Lambda_v$ denote the module of special endomorphsims of $\cA_{\mathcal{O}_{K_v}}[\ell^{\infty}]$. If $\Lambda_v$ has $\bZ_{\ell}$-rank $\leq 1$, then \Cref{onedim} yields the desired result. Therefore, we assume that the rank is at least 2. By \Cref{twodim}, the rank of $\Lambda_v$ is at most two, so we assume that the rank equals two. Let $n = n'_0 + ke_v$, where $n'_0 -n_0 < e_v$ with $n_0$ as in \Cref{numspecial}. We have $M_{v,n} \subset \ell^k \cap M_{v,n_0} + P'_n$, where $P'_n$ is a rank-two sublattice of $M_{v,n_0}$. There is no unique choice of $P'_n$, so we choose $P_n$ to be any one with minimal root-discriminant $d(P_n)$. As $\cB$ has no special endomorphisms generically, it follows that $d(P_n) \rightarrow \infty$. 

We first deal with the case when $d(P_n) \geq \log(N)^2$. 
Since $\mu_1(M_{v,n})\mu_2(M_{v,n})$ is of the same order of magnitude as $d(P_n)$,
then $\#\{v \in M_{v,n}: Q(v) \leq N\}= O(\frac{N^2}{ \ell^{2n/e_v}} + \frac{N^{3/2}}{\ell^{n/e_v}} + \frac{N}{d(P_n)} + N^{1/2})$ by \Cref{dav}.\footnote{Note that $k$ differs from $n/e_v$ by a quantity bounded independent of $n$, so $ 1/\ell^{k} = O(1/\ell^{\frac{n}{e_v}})$.}   This quantity is $o(N/\log N)$ and so the result follows in this case. 

Suppose now that $d(P_n) \leq \log(N)^2$. By \Cref{belowbest} below, if $v \in M_{v,n}$ has norm bounded by $N$, it follows that $v \in P_n$. As $d(P_n) \rightarrow \infty$, the same argument used to finish the proof of \Cref{archBest} applies to prove that the proportion of primes $p$ such that there exists a $\cB \in T_{\fp}(\cA)$ modulo $v^n$ goes to zero. 
\end{proof}

\begin{lemma}\label{belowbest}
Notation as above. Suppose that $d(P_n) \leq \log(N)^2$. If $Q(v) \leq N$, then $v \in P_n$. 
\end{lemma}
\begin{proof}
For brevity set $d = d(P_n)$. Fix a constant $n_0$ as in \Cref{numspecial}.
%We may assume that $n_0 = 1$, as we are only concerned with asymptotes as $n$ tends to infinity. %\ananth{Yunqing, do you have an argument for the next statement?} 
 Let $P'_n$ denote the intersection with $M_{v,n_0}$ of the orthogonal complement of $P_n\otimes \bQ$ in $M_{v,n_0}\otimes \bQ$. %It follows that for any $v \in M_{v,n_0}$, we have that $Dv \in P_n + P'_n$. 
We have that $Cd^4 M_{v,n_0}\subset P_n+P_n'$, where $C$ is a positive constant only depending on the discriminant of  $Q$ on $M_{v,n_0}$. Indeed, let $P_n^\vee$ (resp. $P_n'^\vee$) denote the dual lattice of $P_n$ (resp. $P_n'$) in $P_n\otimes \bQ$ (resp. $P_n'\otimes \bQ$) with respect to the restriction of the quadratic form $Q$ to $P_n\otimes \bQ$ (resp. $P_n'\otimes \bQ$). Then $d^2 P_n^\vee \subset P_n$. On the other hand, there is a constant $C'$ depending only on $\disc Q$ such that $P_n'$ is spanned by two vectors $x,y$ such that $Q(x), Q(y)\leq C'd^2$ (since they are given by solving linear equations with coefficients bounded by $O(d)$). Therefore, there exists a constant $C$ such that $Cd^4 P_n'^\vee\subset P_n'$. Since $M_{v,n_0}\subset (P_n+P_n')^\vee = P_n^\vee+P_n'^\vee$, we have $Cd^4 M_{v,n_0}\subset P_n+P_n'$.

Let $v \in M_{v,n}$ satisfy $Q(v) \leq N$. Suppose that $v = u + \ell^{\lfloor (n-n_0)/e_v\rfloor}w$ with $u \in P_n$ and $w\in M_{v,n_0}$, and let $Cd^4 w = w_1 + w'_1$, where $w_1 \in P_n$ and $w'_1 \in P'_n$. Then, $$Cd^4v = (Cd^4u +  \ell^{\lfloor (n-n_0)/e_v \rfloor}w_1) +  \ell^{\lfloor (n-n_0)/e_v \rfloor}w'_1,$$ and thus $$C^2d^8 N \geq Q(C d^4 v) \geq  \ell^{2\lfloor (n-n_0)/e_v \rfloor} Q(w'_1).$$ As $\ell^{2\lfloor (n-n_0)/e_v \rfloor}> C^2d^8N$, it follows $w'_1 = 0$ as required. 
\end{proof}

\section{Proof of the main theorem}

The goal of this section is to deduce our main theorem from the results in \S\S3-4 which provide upper bounds of the local intersection numbers. Recall that $A$ is an abelian surface over $K$ with an $\fa$-polarization and $\cO_F\subseteq\End(A)$. As in \ref{sub_H&HZ}, we will assume the existence of a semi-abelian scheme $\cA$ over $\cO_K$ with semistable reduction everywhere, whose generic fiber is $A$. Recall that $p$ denotes a prime which is totally split in the narrow Hilbert class field of $F$. To prepare for our proof, we first use Borcherds' theory to choose a suitable Hirzebruch--Zagier divisor in the Hilbert modular surface and then compute the asymptotic of Faltings heights on the Hecke orbits.

\subsection{Borcherds' theory and the Faltings height}\label{sub_Borcherds&ht}
We devote this subsection to applying arithmetic Borcherds' theory to choose a rational section of certain tensor powers of the Hodge line bundle. We then interpret the Faltings height of an abelian surface as a certain Arakelov intersection number.

We use $\cA^{\sa}$ to denote the universal family of semi-abelian schemes over $\bcH$ (with suitable level structure). In \cite{BBK}*{\S6}, the authors explain a way to define the arithmetic intersection independently of the choice of a level structure. We recall their definition in \ref{stacky}. 
Let $e:\bcH\rightarrow \cA^{\sa}$ be the identity section and let $\omega=\det(e^*\Omega^1_{\cA^{\sa}/\bcH})$ over $\bcH$ be the Hodge line bundle. We endow $\omega$ with a Hermitian metric $||\cdot||_F$ (only on $\cH(\bC)$) as in \cite{F85}*{sec.~3} and denote by $\homg$ the Hermitian line bundle with log singularity along the boundary. By definition, we have
$$h_F(A)=ht_{\homg}([\cA]),$$
where $h_F$ denotes the stable Faltings height and $ht$ is the height function of subvarieties of an arithmetic variety with respect to certain arithmetic cycles (see, for example, \cite{BBK}*{\S1.5, eqn.~(1.17)} and we normalize $h_F$ and $ht_{\homg}$ to be independent on the choice of $K$; more specifically, $||\ell||_v=\ell^{-\frac{[K_v:\bQ_\ell]}{[K:\bQ]}}$).

It is well known that the space of global sections of the line bundle $\omega^{\otimes k}$ over $\bcH_{\bC}$ (resp. $\bcH_{\overline{\bQ}}$) is the space of Hilbert modular forms of parallel weight $k$ over $\bC$ (resp. $\overline{\bQ}$);
see, for example, \cite{FC}*{Chp.~V.1} and \cite{Chai}*{sec.~4}.
Up to a constant, the Hermitian metric\footnote{We also use $||\cdot||_F$ to denote the metric on $\homg^{\otimes k}$ given by the tensor product of the Hermitian metric $||\cdot||_F$ on $\homg$.} $||\cdot||_F$ on $\omega^{\otimes k}$ is defined by $||f(z)||_{\text{Pet}}=|f(z_1,z_2)(\Im z_1)^{k/2}(\Im z_2)^{k/2}|$, where $f$ is a Hilbert modular form of parallel weight $k$ and $z=(z_1,z_2)\in \bH^2$: indeed, this follows from the $G(\bR)$-invariance of both metrics. 

\begin{lemma}\label{cpt_div}
There exist a positive integer $k$ and a meromorphic Hilbert modular form $\Psi$ over $\overline{\bQ}$ of parallel weight $k$ such that the divisor $\Div(\Psi)$ defined by $\Psi$ on $\cH_{\bQ}$ is given by $\sum_{r\in \bI} c_r T(r)$, where $c_r\in \bZ$ and $\bI$ is a finite subset of $$\bJ=\{qD \mid \text{q is a rational prime inert in $F$}\}.$$ In particular, $\Div(\Psi)$ is a weighted sum of compact Shimura curves.
\end{lemma}
Borcherds' theory lifts weakly holomorphic modular forms on modular curves to meromorphic Hilbert modular forms on $\cH_{\bC}$. By lifting, it means that the divisor defined the resulting Hilbert modular form is determined by the principal part of the Fourier expansions of the given modular form at cusps of the modular curve. Borcherds and Bruinier showed that the existence such lift of certain modular form can be verified by certain explicit conditions on the Fourier coefficients of its principal part. The Fourier expansions of Borcherds lifts (sometimes also called Borcherds products) have also been studied by many people, which leads to an arithmetic theory of these lifts. One may see \cite{BBK}*{\S4} for a summary of relevant results when the discriminant of $F$ is a prime.
\begin{proof}[Proof of \Cref{cpt_div}]
By \cite{Br16}*{Thm.~1.1},\footnote{\cite{Br16}*{Thm.~1.1} is a generalization of \cite{BBK}*{Lem.~4.11}. The proof of this lemma, which only deals with the case when $D$ is a prime, contains the main idea of the proof for the general case.} in which we take the infinite admissible set to be $\bJ$, there exists a Borcherds product $\Psi'$ of non-zero weight $k$ whose divisor is supported on $\cup_{r\in \bJ}T(r)$. In other words, $\Psi'$ is a Hilbert modular form of parallel weight $k$ over $\bC$. We may assume $k>0$, since otherwise we just take $\Psi'^{-1}$. By \cite{Br16}*{Prop.~3.1}, the weakly holomorphic modular $f$ whose Borcherds lift is $\Psi$ has integral Fourier coefficients. \cite{Hor}*{3.2.14} shows that, after multiplying by a suitable scalar, the Borcherds lift of a modular form with Fourier coefficients in $\bQ$ is defined over $\overline{\bQ}$. In particular, if we take $\Psi$ to be $\Psi'$ multiplied by a suitable scalar, then $\Psi$ is a rational section of $\omega^{\otimes k}$ over $\bcH_{\overline{\bQ}}$. 
For $r\in \bJ$, the divisor $T(r)$ is compact by \Cref{good_r}. 
\end{proof}

\begin{para}\label{vertical_div}
We view $\Psi$ as a rational section of $\omega^{\otimes k}$ over $\bcH_{\cO_{K'}}$, where $K'$ is a large enough number field such that $\Psi$ is defined. Hence $\Div(\Psi)=\sum_r c_r \cT(r) + \sum_p \cE_p$, where the second sum is over finitely many $p$ and $\cE_p$ is a finite (weighted) sum of irreducible components of $\bcH_{\overline{\bF}_p}$.
\end{para}

\begin{para}\label{stacky}
Given an arithmetic divisor $\cD$ and a horizontal $1$-cycle $\cZ$ intersecting properly on $\bcH_{\cO_{K'}}$, one defines the (arithmetic) intersection number as follows: (see, for example, \cite{BKK}*{Thm.~1.33}\footnote{The higher tor group vanishes since we work with a Cartier divisor which intersects the $1$-cycle properly.} for regular schemes and \cite{Yang}*{eqn.~(2.1)} for regular Deligne--Mumford stacks.)
$$\cD.\cZ=\sum_v\sum_{x\in (\cZ\cap\cD)(\overline{\bF}_v)}\frac{\log (\# \tilde{\cO}_{\cZ\cap \cD,x})}{\# \Aut(x)}=\sum_v\sum_{x\in (\cZ\cap\cD)(\overline{\bF}_v)}\frac{ \Length (\tilde{\cO}_{\cZ\cap \cD,x}) \log (\# k(x))}{\# \Aut(x)},$$
where $v$ ranges over the through finite places of $K'$, the intersection $\cZ\cap\cD=\cZ\times_{\bcH}\cD$ is a Deligne--Mumford stack of dimension $0$, the ring $\widetilde{\cO}_{\cZ\cap \cD,x}$ is the strictly Henselian local ring of $\cZ\cap \cD$ at $x$, and $k(x)$ is the residue field of $\widetilde{\cO}_{\cZ\cap \cD,x}$. By definition, $\displaystyle \frac{\cD.\cZ}{[K':\bQ]}$ is independent of the choice of $K'$.

In \cite{BBK}*{sec.~6.3}, they define the arithmetic intersection number on $\bcH$ as the arithmetic intersection number of the pull back of arithmetic cycles to $\bcH(N)$, the Hilbert modular surface with full level $N$-structure with $N\geq 3$, divided by the degree of the map $\bcH(N)\rightarrow \bcH$. This is the idea behind the above formula.
\end{para}

\begin{rem}\label{multiplicity}
For $\cD$ and $\cZ$ as above, let $n$ denote the largest integer such that $\cZ$ is contained in $\cD$ modulo $v^n$ (here, we consider $\cZ$ and $\cD$ as subschemes of the coarse Hilbert modular surface). In our applications, we will only consider the intersection at finitely many places. Therefore, we may pass to a suitable level structure \'etale at these finitely many places so that $\cT(r)$ are regular (\cite{Ca}). Then the length at $v$ referred to in \Cref{stacky} differs from $n$ by an absolutely bounded factor. As we are only concerned with bounds, we will in the sequel restrict ourselves with controlling the growth of $n$.
\end{rem}

\begin{lemma}\label{ht_compare}
Assume that $\End(A_{\overline{K}})=\cO_F$. Following the notation as in \Cref{cpt_div}, there exists a constant $C_1$ independent of $A$ such that 
$$\left | h_F(A)-\frac 1 {k[K:\bQ]} [\cA]. \sum_{r\in \bI} c_r\cT(r) -\frac 1 {k[K:\bQ]} \sum_{\sigma:K\inj\bC}\frac 1{\# \Aut(A_{\overline{K}})}\log ||\Psi(\sigma([A]))||_{\rm{Pet}} \right | <C_1.$$
\end{lemma}

\begin{proof}
By \Cref{HZ=spEnd}, if $[A]$ lies on $T(r)$, then $\End(A_{\overline{K}})$ is strictly larger than $\cO_F$. Hence our assumption implies that $[A]$ does not lie on any $T(r)$. It follows that the $1$-cycle $[\cA]$ intersects $\sum_{r\in \bI} c_r T(r)$ properly. Set $\cE(A)=\frac 1{[K:\bQ]}[\cA].(\sum_p\cE_p)$, where $\sum_p\cE_p$ was defined in \ref{vertical_div}. By definition, $\cE(A)$ is bounded by an absolute constant independent of $A$. 
Moreover,
$$h_F(A)=ht_{\homg}([\cA])=\frac 1 {k[K:\bQ]} [\cA]. \sum_{r\in \bI} c_r\cT(r) -\frac 1 {k[K:\bQ]} \sum_{\sigma:K\inj\bC}\frac 1{\# \Aut(A_{\overline{K}})}\log ||\Psi(\sigma([A]))||_F+ \frac 1 k \cE(A),$$
(see, for example, \cite{Yang}*{eqn.~(2.3)}). The lemma then follows from the fact that $||\cdot||_F$ and $||\cdot||_{\text{Pet}}$ differ by an absolute constant independent of $A$.
\end{proof}

%By \Cref{ht_compare}, the intersection of a Hecke orbit with $\sum_{r\in \bI} c_r\cT(r)$ is related to the sum of the Faltings heights of the abelian surfaces that correspond to the points of $\cH_{\overline{\bQ}}$ in this Hecke orbit.
We end this subsection with a formula for the average Faltings height of abelian surfaces corresponding to points in $T_{\fp}[A]$ when $A$ has good reduction at all the primes of $K$ above $p$. The idea of proof builts on use Autissier's idea in \cite{Au}. From now on, we say that $A$ has \emph{good reduction at $p$} (resp. \emph{ordinary reduction at $p$}) if $A$ has good reduction (resp. ordinary reduction) at all the primes of $K$ above $p$. 

\begin{proposition}\label{thm_glb}
Let $p$ be a prime as in \cref{phecke}. If $A$ has good reduction at $p$, then 
 $$ \sum_{[B]\in T_{\fp}[A]}h_F(B)=(p+1)h_F(A)+\frac{p-1}2 \log p.$$
\end{proposition}

\begin{proof}The proof consists two parts. We first show that $\sum_{[B]\in T_{\fp}[A]}h_F(B)-(p+1)h_F(A)$ is independent of $A$. Then we compute this quantity in a particular case.
\begin{enumerate}
\item
Let $\cH(\fp)$ denote the Hilbert modular surface over $\bZ$ with $\Gamma_0(\fp)$ level structure. The stack $\cH(\fp)$ parameterizes degree $p$ isogenies $\phi:A_1\rightarrow A_2$ between abelian varieties with $\cO_F$-multiplication such that $\ker(\phi)\subset A_1[\fp]$ (see for example \cite{Pa}*{sec.~2.2}). 
Let $\pi_i:\cH(\fp)_{\bZ_{(p)}}\rightarrow \cH_{\bZ_{(p)}}$ for $i=1,2$ be the forgetful map that sends $\phi$ to $A_i$.
We claim that each $\pi_i$ is finite flat. 

We first show that $\pi_i$ is quasi-finite. Let $v$ be any finite place of $K$ over $p$. The group scheme $\cA[\fp^{\infty}]$ of $\fp$-power torsions of $\cA$ is a $p$-divisible group of height 2, whose mod $v$ reduction has dimension 1. There are two cases: the mod-$v$ reduction of $\cA[\fp^{\infty}]$ is either ordinary, or supersingular. If ordinary, \cite{FC}*{\S VII.4} shows that there are only finitely many degree $p$ subgroups of the mod-$v$ reduction of $\cA[\fp]$. Therefore, by the modular interpretation of $\cH(\fp)$, the map $\pi_1$ is quasi-finite. Now we assume that the mod-$v$ reduction of $\cA[\fp^{\infty}]$ is supersingular. Since the number of degree $p$ subgroups of the reduction of $\cA[\fp]$ only depends on the isomorphism class of the $p$-divisible group, we may assume that $\cA$ has supersingular reduction at $v$. Since $p$ is split in $F$, the supersingular locus of $\cH_{\bF_v}$ is $0$-dimensional and hence there are only finitely many mod-$v$ points on $\cH$ corresponding to abelian surfaces isogenous to $\cA$ mod $v$. In particular, $\pi_1$ is quasi-finite. The same argument applies to $\pi_2$ when we study the kernel of the Rosati involution of $\phi$.

By \cite{Pa}*{2.1.3, Cor.~2.2.3}, the stack $\cH(\fp)$ is Cohen--Macaulay and $\cH$ is regular. Since all the fibers of $\pi_i$ are $0$-dimensional, then by \cite{EGAIV}*{II.6.1.5}, each $\pi_i$ is flat. On the other hand, the  $\pi_i$ are proper by \cite{Pa}*{the discussion after Def.~2.2.1}. Therefore, each $\pi_i$ is finite flat.

By the argument in \cite[Theorem 5.1]{Au}, one observes that the independence of the quantity $\sum_{[B]\in T_{\fp}[A]}h_F(B)-(p+1)h_F(A)$ on $A$ is a formal consequence of the finite-flatness of $\pi_i$, the normality of $\cH$, and the irreducibility of $\cH_{\bF_v}$ for every $v$ above $p$. 

\item
We now compute $\sum_{[B]\in T_{\fp}[A]}h_F(B)-(p+1)h_F(A)$ when $A$ has ordinary reduction at $p$.\footnote{One may also choose any CM abelian surface on $\cH$ and apply the formula for the Faltings height in \cite{lucia} to compute this difference.} Such an $A$ always exists: indeed, for a CM field $K_2$ containing $F$ such that $K_2$ is Galois over $\bQ$ of degree $4$ and $p$ splits completely in $K_2$, there exist abelian surfaces with CM by $K_2$ that correspond to points on $\cH$ and these abelian surfaces are ordinary at $p$.

Now assume that $A$ has good ordinary reduction at $p$ and we prove the result for such $A$. We first enlarge $K$ so that all $[B]\in T_{\fp}[A]$ are defined over $K$. 
Since $\cA[\fp^\infty]$ over $\cO_{K_v}$ is $1$-dimensional, it only contains a unique degree $p$ subgroup which is multiplicative (equivalently, in the ordinary case, not \'etale) for any $v$ above $p$. 
Then by \Cref{Tpmoduli}, there exists only one element in $T_{\fp}[\cA]$ that corresponds to an isogeny with multiplicative kernel. 
Now we apply \cite{F85}*{Lem.~5}. By the standard calculation on $\Omega^1$ of finite flat groups of degree $p$, at each $v$ there are $p$ out of $p+1$ elements in $T_{\fp} [A]$ such that the term $\log(\#e^*(\Omega^1_{\ker \phi/\cO_{K_v}}))$ in Faltings' formula is $0$, and one element such that $\log(\#e^*(\Omega^1_{\ker \phi/\cO_{K_v}}))=-[K_v:\bQ_p]\log p$. We obtain the desired formula by summing up all the local contributions. \qedhere
\end{enumerate}
\end{proof}

\subsection{Proof of \Cref{main}}\label{proof}
\begin{para}\label{sketch}
We first sketch the proof of \Cref{main}. First, we use \Cref{main} to choose a good Hirzebruch--Zagier divisor $\sum_{r\in \bI}c_r\cT(r)=\Div(\Psi)$. By \Cref{thm_glb}, we have $\sum_{[B]\in T_\fp([A])}h_F(B)$ is $O(p\log p)$. On the other hand, the local results in \S\S\ref{finite}-\ref{arch} show that each local term in \Cref{ht_compare} is $o(p\log p)$. 
By local term, we mean either $-\sum_{[B]\in T_{\fp}[A]}\log ||\Psi(\sigma([B]))||_\pet$ for all $\sigma:K\hookrightarrow \bC$ or the $v$-adic intersection number $(\sum_{[\cB]\in T_{\fp}[\cA]}[\cB],\sum_{r\in \bI}c_r\cT(r))_v$ for finite places $v$. This implies that $T_\fp([\cA])$ intersects $\sum_{r\in \bI}c_r\cT(r)$ at infinitely many places as $p\rightarrow \infty$ and then \Cref{main} follows from \Cref{notsimple}.

\end{para}

\begin{para}\label{reduction}
If $\cO_F\subsetneq\End(A_{\overline{K}})$, then by the classification of the endomorphism ring of absolutely simple abelian surfaces over a characteristic zero field, $\End(A_{\overline{K}})\otimes \bQ$ is either an indefinite quaternion algebra over $\bQ$ or a degree $4$ CM field. In the first case, $\cA_{\overline{\bF}_v}$ is not simple if the quaternion algebra splits at ${\rm{char } }\, \bF_v$. In the second case, there exists a positive density set of primes $\ell$ so that $\cA_{\bF_v}$ is supersingular for all $v|\ell$ and hence not geometrically simple. Therefore, to prove \Cref{main}, we assume that $\End(A_{\overline{\bQ}})=\cO_F$ from now on and hence for any $[B]\in T_\fp([A])$, we also have $\End(B_{\overline{K}})=\cO_F$. Therefore, all $T_{\fp}([\cA])$ intersect Hirzebruch--Zagier divisors properly.
\end{para}

\begin{proof}[Proof of \Cref{main}]
We now show that there are infinitely many primes $v$ of $K$ such that $\cA_{\overline{\bF}_v}$ is not simple.
Assume, for the sake of contradiction, that there is a finite set of places $\Sigma$ of $K$ such that $\cA$ has geometrically simple, or bad reduction modulo $v$ for $v\notin \Sigma$. \Cref{good_r} and \Cref{cpt_div} give a meromorphic Hilbert modular form $\Psi$ such that $\Div(\Psi)$ is a compact special divisor $\sum_{r\in \bI} c_r\cT(r)$ with $D|r$ for all $r\in \bI$. The intersection $(T_{\fp}([\cA]),\sum_{r\in \bI}c_r\cT(r))$ has a nonzero $v$-adic term only when $v\in \Sigma$ by \Cref{notsimple}.
Throughout the proof, $p$ will denote a prime which is totally split in the narrow Hilbert class field of $F$ and $v\nmid p$ for all $v\in \Sigma$. We now explain how to choose an increasing sequence of primes $p$ such that we can bound the local terms as described in \ref{sketch} by $\epsilon p \log p$ for arbitrary $\epsilon>0$. 

By \Cref{equinon,bestnonarch} and \Cref{multiplicity}, outside a density-zero set of primes $p$, the $v$-adic intersection  
$$\Big(\sum_{[\cB]\in T_{\fp}[\cA]}[\cB],\sum_{r\in \bI}c_r\cT(r)\Big)_v \leq C_1 \bigg( \sum_{r\in \bI}|c_r| \cdot \Big((p+1)(3e_v\log(2\log p))+\epsilon p (2e_v\log p+1))\Big)\bigg),$$ 
where $C_1$ is the absolute constant mentioned in \Cref{multiplicity}. Notice that $2\log p\geq \log N$.

Let $C_2$ be the density of primes splitting completely in the narrow Hilbert class field of $F$. By taking $\epsilon_1=\epsilon,\,\epsilon_2=\frac{C_2}{4[K:\bQ]}$ in \Cref{thm_arch}, we have that, for $N\gg 0$ and for $p\in [N^{1/2},N]$ in a set of density at least $3C/4$, for all $\sigma:K\inj \bC$, the archimedean term $$-\sum_{[B]\in T_{\fp}[A]}\log ||\Psi(\sigma([B]))||_\pet<\epsilon p\log p.$$

We have shown that for $N\gg 0$, there exists a positive density set of primes $p\in [N^{1/2}, N]$ such that all the local terms are $o(p\log p)$. On the other hand, by \Cref{thm_glb}, $\displaystyle{\sum_{[B]\in T_\fp([A])}h_F(B)}$ has order of magnitude $p\log p$. We then obtain the desired contradiction by applying \Cref{ht_compare} to all $[B] \in T_\fp([A])$.
\end{proof}

\end{document}